%% file: Spectrum_of_RSC_in_Thermo_Regime.tex
\documentclass[oneside,reqno]{amsart}
\usepackage{amssymb}
\usepackage{amsmath}
\usepackage{amsthm}
\usepackage{amsbsy}
\usepackage{bbm}
\usepackage{enumerate}
\usepackage{bm}
\usepackage{hyperref}
\date{\today}
\usepackage{cite}
\usepackage{array}
\usepackage{xcolor}
\usepackage{tikz}
\usepackage{graphics}
\usepackage{subcaption}
\usepackage{multirow}
\usepackage{mathtools}
\usepackage{tikz}
\usepackage{caption}
\usepackage{subcaption}
\usetikzlibrary{shapes,snakes}
\makeatletter
\newcommand{\distas}[1]{\mathbin{\overset{#1}{\kern\z@\sim}}}

\newcommand{\distras}[1]{
  \mathbin{\overset{#1}{\kern\z@\resizebox{\wd\mybox}{\ht\mysim}{$\sim$}}}
}

\input{mypream}

\title [On the spectrum of Random Simplicial Complexes in Thermodynamic Regime] {On the spectrum of Random Simplicial complexes in Thermodynamic Regime}
\author{Kartick Adhikari}
\address{Department of Mathematics\\
 Indian Institute of Science Education and Research Bhopal,\\
  Bhauri, Bhopal, Madhya Pradesh 462066, India }
\email{kartick [at] iiserb.ac.in}

\author{Kiran Kumar A.S.}
\address{Department of Mathematics\\
        Indian Institute of Technology Bombay\\
         Powai, Mumbai, Maharashtra 400076, India}
 \email{kiran [at] math.iitb.ac.in}

\author{Koushik Saha}
\address{Department of Mathematics\\
	Indian Institute of Technology Bombay\\
	Powai, Mumbai, Maharashtra 400076, India}

\email{koushik.saha [at] iitb.ac.in}

\begin{document}

\begin{abstract}

Linial-Meshulam complex is a random simplicial complex on $n$ vertices with a complete $(d-1)$-dimensional skeleton and $d$-simplices occurring independently with probability $p$. Linial-Meshulam complex is one of the most studied generalizations of the Erd\H{o}s-R{\'e}nyi random graph in higher dimensions.

In this paper, we discuss the spectrum of adjacency matrices of the Linial-Meshulam complex when $np \rightarrow \lambda$. We prove the existence of a non-random limiting spectral distribution(LSD) and show that the LSD of signed and unsigned adjacency matrices of Linial-Meshulam complex are reflections of each other. We also show that the LSD is unsymmetric around zero, unbounded and under the normalization $1/\sqrt{\lambda d}$, converges to standard semicircle law as $\lambda \rightarrow \infty$.

In the later part of the paper, we derive the local weak limit of the line graph of the Linial-Meshulam complex and study its consequence on the continuous part of the LSD.
\end{abstract}

\maketitle


\section{Introduction and Main Results}

Random graphs is a major area of research in modern combinatorics. The study of random graphs was initiated by Paul Erd\H{o}s and Alfr{\'e}d R{\'e}nyi in early 1960s \cite{Erdos_renyi_59},\cite{Erdos_Renyi60}. The model they introduced, later known as Erd\H{o}s-R{\'e}nyi random graph, is one of the most important models of random graphs. In network science, random graphs are popular models to study interactions between two nodes. To capture the interactions between more than two nodes, the generalization to simplicial complexes becomes helpful. With the growth in topological data analysis, simplicial complexes are increasingly being used to model real-world systems \cite{TDA_Simplicial},\cite{Brain_Nature_Neuro},\cite{Brain_Sci_Rep}. For systems with inherent randomness, random simplicial complexes are thus useful models \cite{Courtney_Ginestra}.

Linial-Meshulam complex was developed by Linial, Meshulam and Wallach in \cite{Linial_Meshulam_2006},\cite{Meshulam_wallach}. Linial-Meshulam complex  is the earliest random simplicial complex studied in mathematics literature and is a generalization of Erd\H{o}s-R{\'e}nyi random graphs to higher dimensions $d \geq 2$. A Linial-Meshulam complex, denoted by $Y_d(n,p)$, is a random $d$-dimensional complex on $n$ vertices with a complete $(d-1)$-dimensional skeleton and $d$-cells occurring independently with probability $p$. Linial-Meshulam complex has attracted considerable attention from the mathematical community and is one of the most extensively studied generalizations of Erd\H{o}s-R{\'e}nyi random graph \cite{nikolas_Pry},\cite{antti_ron},\cite{linial_peled}.

For an $n \times n$ symmetric matrix $M_n$ with eigenvalues $\lambda_1,\lambda_2,\ldots, \lambda_n$, the empirical spectral distribution of $M_n$ is defined as
\begin{equation}
	F_{M_n}(x)=\frac{1}{n} \sum_{i=1}^{n} \mathbf{1}(\lambda_i \leq x).
\end{equation} 
The probability measure on the real line corresponding to $F_{M_n}$ is known as the  empirical spectral measure of $M_n$, and is given by
\begin{equation}\label{defn:spec_dist}
	\mu_{M_n}=\frac{1}{n} \sum_{i=1}^{n} \delta_{\lambda_i}
\end{equation}
 For a random matrix $M_n$, both the empirical spectral distribution and the empirical spectral measure are random quantities. For random matrices $M_n$, we define the expected empirical spectral distribution(EESA) as the expectation of the empirical spectral distribution, i.e. $\mathbb{E}F_{M_n}$.
For a sequence of matrices $(M_n)_{n \in \N}$, the limit of the empirical spectral distributions $F_{M_n}$ as $n \rightarrow \infty$, is known as the limiting spectral distribution (LSD) of the sequence $(M_n)$.
In this paper, we study the limiting spectral distribution of the adjacency matrices of Linial-Meshulam complex.


We consider two types of adjacency matrices for simplicial complexes, signed adjacency matrix and unsigned adjacency matrix. The signed adjacency matrix is related to high dimensional expanders and combinatorial Laplacians. The unsigned adjacency matrix is related to random walks on random simplicial complexes \cite{ori_Spectrum_homology} and the energy of the random simplicial complexes\cite{Knill_energy}. Adjacency matrices are also related to different notions of centrality measures of simplicial complexes such as simplicial degree centrality and eigenvector centrality \cite{Gomez_Centality}.

 For $np \rightarrow \infty$ it is known that the empirical spectral distribution of the adjacency matrix of Erd\H{o}s-R{\'e}nyi random graph converges weakly to the standard semi-circle law under the normalization ${1}/{\sqrt{np(1-p)}}$\cite{Furedi_Komlos}.  In \cite{antti_ron}, Knowles and Rosenthal studied the spectral distribution of the signed adjacency operator of $Y_d(n,p)$ for $np(1-p) \rightarrow \infty$ and showed that the limiting spectral distribution under the scaling ${1}/{\sqrt{np(1-p)}}$ is the standard semi-circle law. Recently, it was observed in \cite{Leibzirer_Ron} that the empirical spectral distribution of the centred version of unsigned adjacency operator also converge to the standard semi-circle law under the scaling $1/{\sqrt{np(1-p)}}$ when $np(1-p) \rightarrow \infty$. The limiting spectral distribution of the Lagrangian for $np \rightarrow \lambda>0$ was studied in \cite{Shu_Kanazawa_Betti}, for much larger class of random simplicial complexes.

For $np \rightarrow \lambda >0 $, the limiting spectral distribution, $\nu_\lambda$, of the adjacency matrix of Erd\H{o}s-R{\'e}nyi random graph was first studied in \cite{Bauer_golinelli}. Later analytical results have been obtained in \cite{Zakharevich},\cite{Jung_Lee}. The exact limiting distribution is not known in this case, but several interesting properties of the limiting distribution are known. For example, it is known that for every algebraic number $x \in \R$, its mass $\nu_{\lambda}(\{x\})$ is strictly greater than 0 \cite{Salez_algebraic}. It was proved in \cite{BVS} that $\nu_\lambda$ has a continuous part if and only if $\lambda$ is greater than 1, and a non-vanishing absolutely continuous part for $\nu_{\lambda}$ at zero emerges for $\lambda=e$ \cite{Coste_Salez}.

In this paper, we study the limiting spectral distribution of signed and unsigned adjacency matrices of Linial-Meshulam complex when $np \rightarrow \lambda \in (0, \infty)$. Our first main result proves the existence of a non-random LSD for these matrices (see Theorem \ref{thm: existence_LSD}). We further show that the LSD of signed and unsigned adjacency matrices are reflections of each other as $np \rightarrow \lambda$. This property is in general not true for individual simplicial complexes or Linial-Meshulam complexes on finitely many vertices.

\begin{figure}[h]
	\includegraphics[height=65mm, width =165mm]{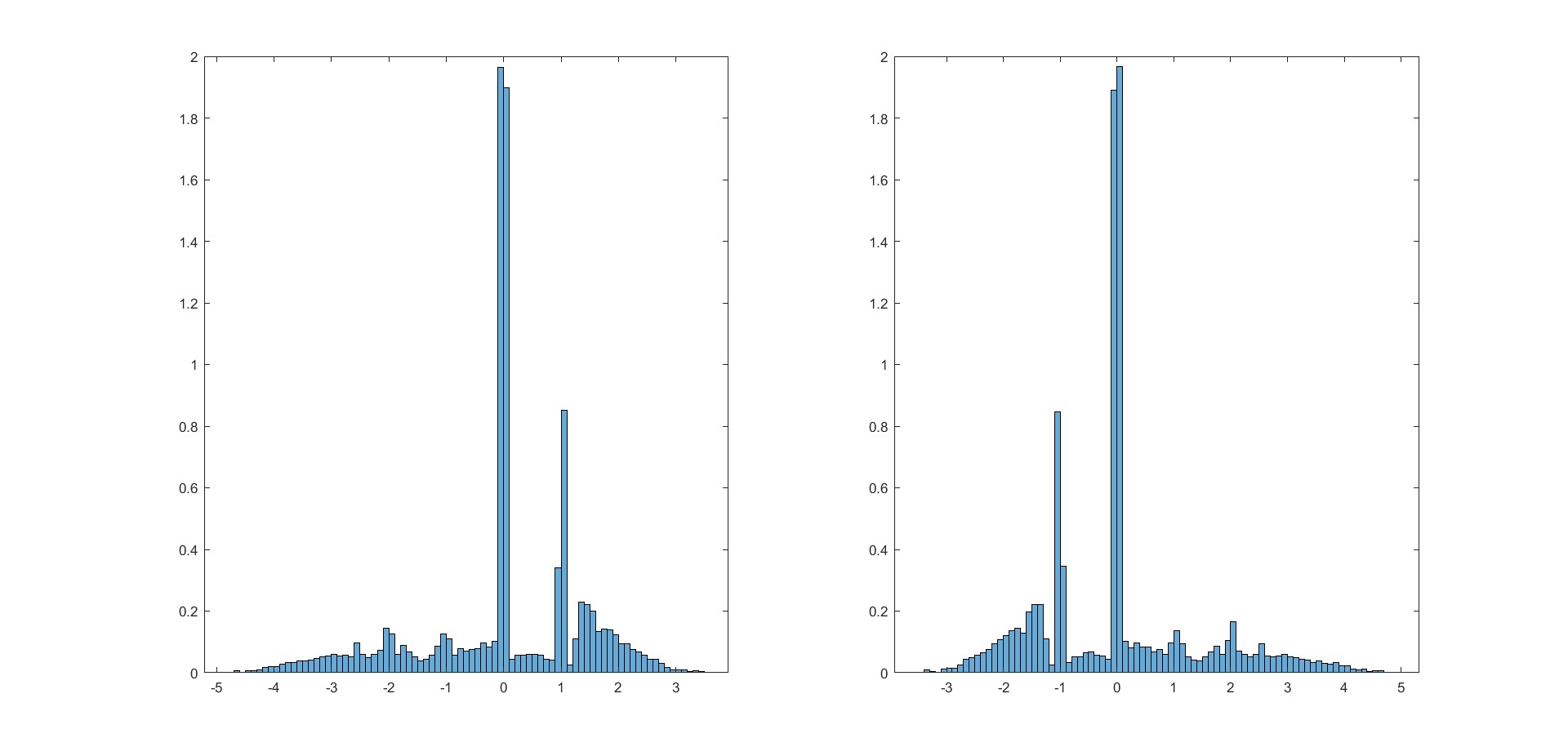}
	\bigbreak
	\includegraphics[height=65mm, width =165mm]{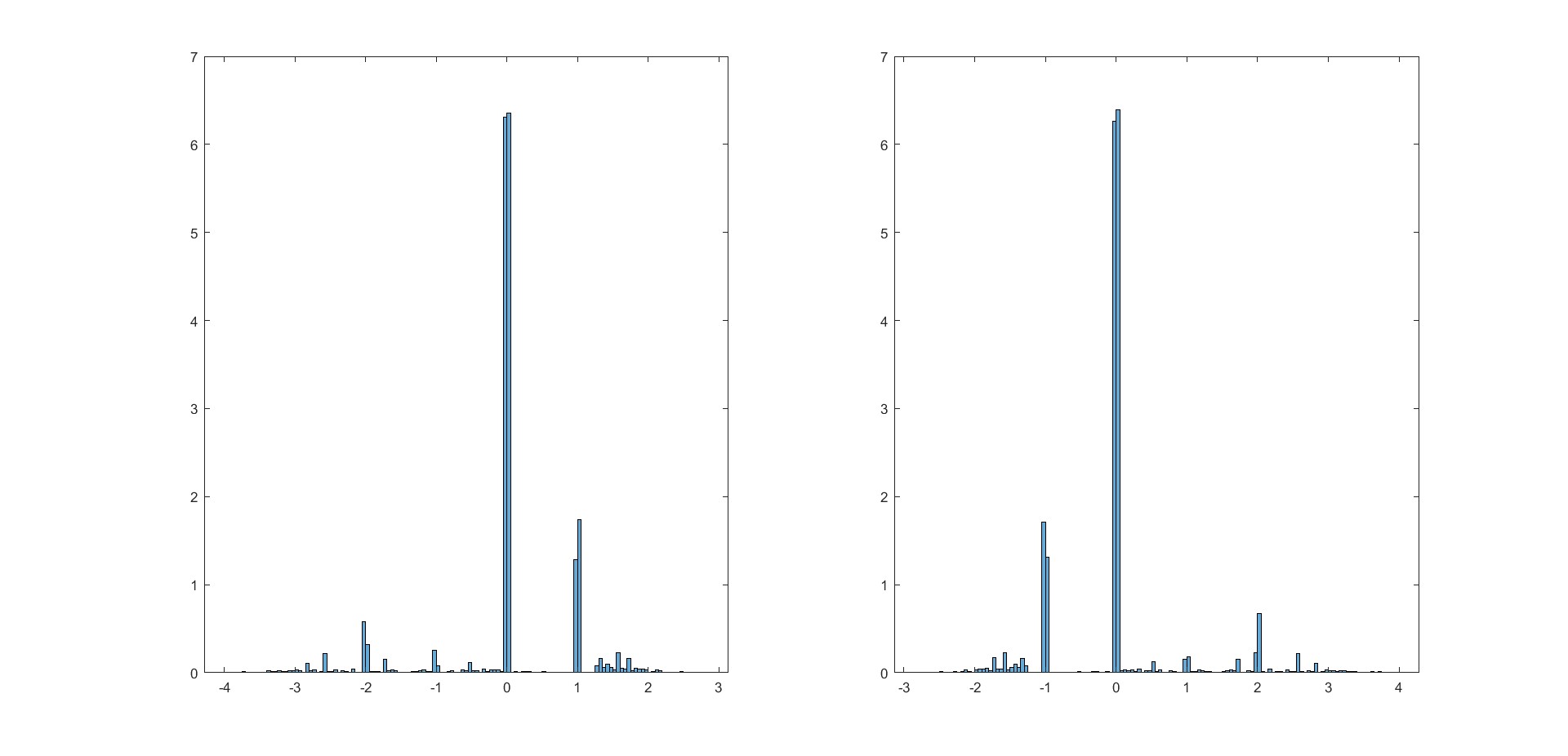}
	\caption{(above) The normalized histograms for eigenvalues of unsigned (left) and signed (right) adjacency matrices of Linial-Meshulam complex for $(d,n,\lambda)=(2,100,1)$. (below) The normalized histograms for eigenvalues of unsigned (left) and signed (right) adjacency matrices of Linial-Meshulam complex for $(d,n,\lambda)=(2,100,0.5)$.}\label{fig:limiting_dist.jpg}
\end{figure}

In connection with this, we prove several interesting properties of the limiting spectral distribution associated to moments. It is known that the LSD, $\mu_\lambda$, of Erd\H{o}s-R{\'e}nyi random graphs for the cases $np \rightarrow \lambda$ under rescaling converge to the Wigner semi-circle law as $\lambda \rightarrow \infty$ \cite{Nathenel_Laurent}, \cite{Jung_Lee}. In this paper, we show that the same conclusion holds for Linial-Meshulam complex as well, for both signed and unsigned adjacency matrices (See Proposition \ref{prop:Gamma_semicircle}). We also obtain the almost sure limit for the Frobenius norm of the signed and unsigned adjacency matrices. 

Theorem \ref{thm: existence_LSD} implies that the limiting spectral distribution of signed adjacency matrix can be completely determined by the limiting distribution of the unsigned adjacency matrix. In section \ref{sec: local_weak}, we study a graph associated with Linial-Meshulam complex, known as the line graph of the Linial-Meshulam complex. The line graph has same adjacency matrix as the Linial-Meshulam complex and thus the techniques from graph theory can be used to study the simplicial complexes as well. In particular, we use analytical techniques on graphs to get insights about $\Gamma_d(\lambda)$.

 The concept of local weak convergence introduced independently by Benjamin and Schramm \cite{Benjamin_Schramm_lwc}, and Aldous and Steele \cite{Aldous_Steele}, is a valuable tool in the study of random graphs. The central idea here is to study the probability measure induced by a weighted random graph (random network) on the metric space of unlabelled rooted locally finite networks. From the weak limit of the probability measure, local properties of the random structure for large size is obtained. For detailed discussions on this topic, see \cite{Lyon_Aldous} or \cite{remco_van}. In last decade, local weak convergence has evolved as an effective tool to study properties of the limiting spectral distribution of sparse Erd\H{o}s-R{\'e}nyi random graphs \cite{Bordenave_Lelarge_Salez},\cite{Salez_algebraic}.


For $p=\lambda/n$ and $n \rightarrow \infty$, the local weak limit of the Erd\H{o}s-R{\'e}nyi graph is the Galton-Watson tree with offspring distribution $\operatorname{Poi}(\lambda)$\cite{Dembo_Montanari}. In this paper, we generalize this result to higher dimensions. We show that for $d \geq 2$, the local weak limit of the line graph of Linial-Meshulam complex is a generalization of Galton-Watson tree, which we call the $d$-block Galton-Watson graph (see Theorem \ref{thm: weak limit}).

In \cite{linial_peled}, Linial and Peled had considered a bipartite graph associated to the boundary operator of Linial-Meshulam complex to derive thresholds for $d$-collapsibility and vanishing of $d-$th homology. The crucial part of our work is to connect the unlabelled bipartite graph to the line graph through a measurable map. The advantage here is that the limiting measure $dGW(d\operatorname{Poi}(\lambda))$, is a unimodular measure. As a corollary of this, we show that for $\lambda \leq 1/d$, the limiting spectral distribution of adjacency matrices is purely atomic.

\section{Preliminaries and Main Results}
In this section, we introduce the necessary preliminaries and state our main results. For convenience in presentation, we divide this section into two subsections. In subsection \ref{sub:LSD of matrices}, we introduce the necessary combinatorics behind random simplicial complexes and state the results on the existence of LSD of adjacency matrices of $Y_d(n,p)$ and some of its properties. In subsection \ref{subsec:line grap}, we state the main results concerning the line graph of $Y_d(n,p)$.
\subsection{LSD of adjacency matrices of $Y_d(n,p)$}\label{sub:LSD of matrices}
\begin{definition}\label{defn: d-cell}
	Let $V$ be a finite set. A simplicial complex $X$ with vertex set $V$ is a collection $X \subset \mathcal{P}(V)$ such that if $\tau \in X$ and $\sigma \subset \tau$, then $\sigma \in X$. An element of the simplicial complex is called a cell. The dimension of a cell $\sigma \in X$ is defined as $|\sigma|-1$, and an element of dimension $j$ is called a $j-$cell. The dimension of a non-empty simplicial complex is defined as the maximum of the dimensions of its elements.
	For a simplicial complex $X$ and integer $j \geq -1$, we denote the set of all $j-$cells in $X$ by $X^j$.
\end{definition}

We remark here that all non-empty simplicial complexes contain the null set and the null set is the only cell of dimension $-1$. For a simplicial complex $X$, the cells of dimension 0 are the singletons $\{v\} \in X$ such that $v$ is an element of $V$.

For a vertex set $V$ and $j \geq -1$, we define the complete $j$-dimensional complex as the set of all subsets of $V$ with cardinality less than or equal to $j+1$. Another related object is the $\ell$-dimensional skeleton. Let $X$ be a simplicial complex on $V$. We say $X$ has a complete $\ell$-dimensional skeleton if $X$ contains all subsets of $V$ of cardinality less than or equal to $\ell+1$. Throughout this paper, unless stated otherwise, the vertex set taken is $V=\{1, 2, \ldots , n\}$.

An Erd\"{o}s-Renyi graph on $n$ vertices is the random graph on the vertex set $V=\{1,2,\ldots,n\}$ constructed by adding edges independently with probability $p$. In other words, Erd\"{o}s-Renyi graph is the 1-dimensional random simplicial complex with a complete $0$-dimensional skeleton and $1$-cells occurring independently with probability $p$. Linial-Meshulam model is a generalization of this idea to higher dimensional simplicial complexes.

\begin{definition}\label{defn:Linial-Meshulam Model}
	For natural numbers $d \geq 2,n \geq d+1$ and $p \in [0,1]$, $Y_d(n,p)$ is called a random simplicial complex, if $Y_d(n,p)$ has a complete $(d-1)$-dimensional skeleton and $d$-cells occurring independently with probability $p$. When $d,n,p$ are clear from the context, we shall use $Y$ to denote $Y_d(n,p)$.
\end{definition}

The random simplicial complex defined above is known as the Linial-Meshulam complex. Several other variations of random simplicial complexes have also been studied in the literature. For brief surveys on different types of random simplicial complexes and connections between them, see \cite{Kahle_survey}, [Chapter 22\cite{Random_complex_survey}] or \cite{Bobrowski_survey}.

We now proceed to define orientation on a $j$-cell. An orientation is an ordering of the vertices of the $j$-cell, and an oriented cell is represented by a square bracket. Two orientations $[x_0,x_1,\ldots ,x_j]$ and $[y_0,y_1,\ldots , y_j]$ are said to be equal if the permutation $g$ given by $g(x_i)=y_i$ is an even permutation. Thus for $j \geq 1$, each $j$-cell has exactly two orientations. For an oriented $j$-cell $\sigma$, we use $\overline{\sigma}$ to denote the same $j$-cell with the opposite orientation. As an example, for the $j$-cell $\{  x, y, z\}$, $[x,y,z]=[y,z,x]=[z,x,y]$ and $[x,z,y]=[z,y,x]=[y,x,z]$ are the two different orientations.

Suppose $V$ is an ordered set. The ordering on $V$ induces an ordering on each $j$-cell. The orientation corresponding to this ordering is called the positive orientation, while the other orientation is called the negative orientation. The set of positively oriented $j$-cells of $X$ is denoted by $X_+^j$ and the set of all oriented $j$-cells is denoted by $X_{\pm}^j$. For $j \geq 1$, we have $|X_{\pm}^j|=2|X^j|$.


\begin{figure}
	\captionsetup[subfigure]{justification=centering}
	\hspace{-10mm}
	\begin{subfigure}[b]{0.45\textwidth}
		\centering
		\includegraphics[height=45mm, width=\textwidth]{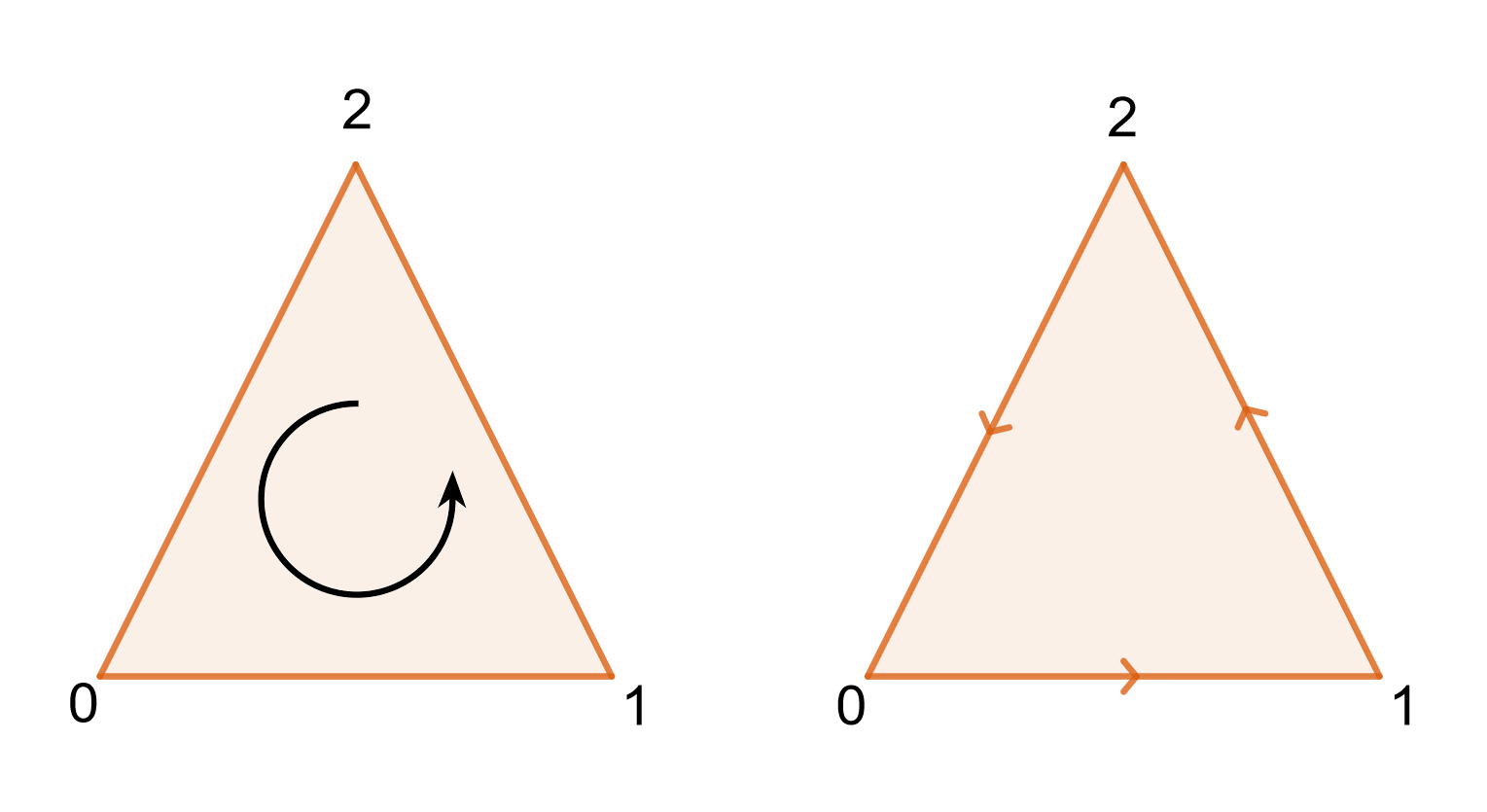}
		\caption{}
	\end{subfigure}
	\hspace{10mm}
	\begin{subfigure}[b]{0.45\textwidth}
		\centering
		\includegraphics[height=45mm, width=\textwidth]{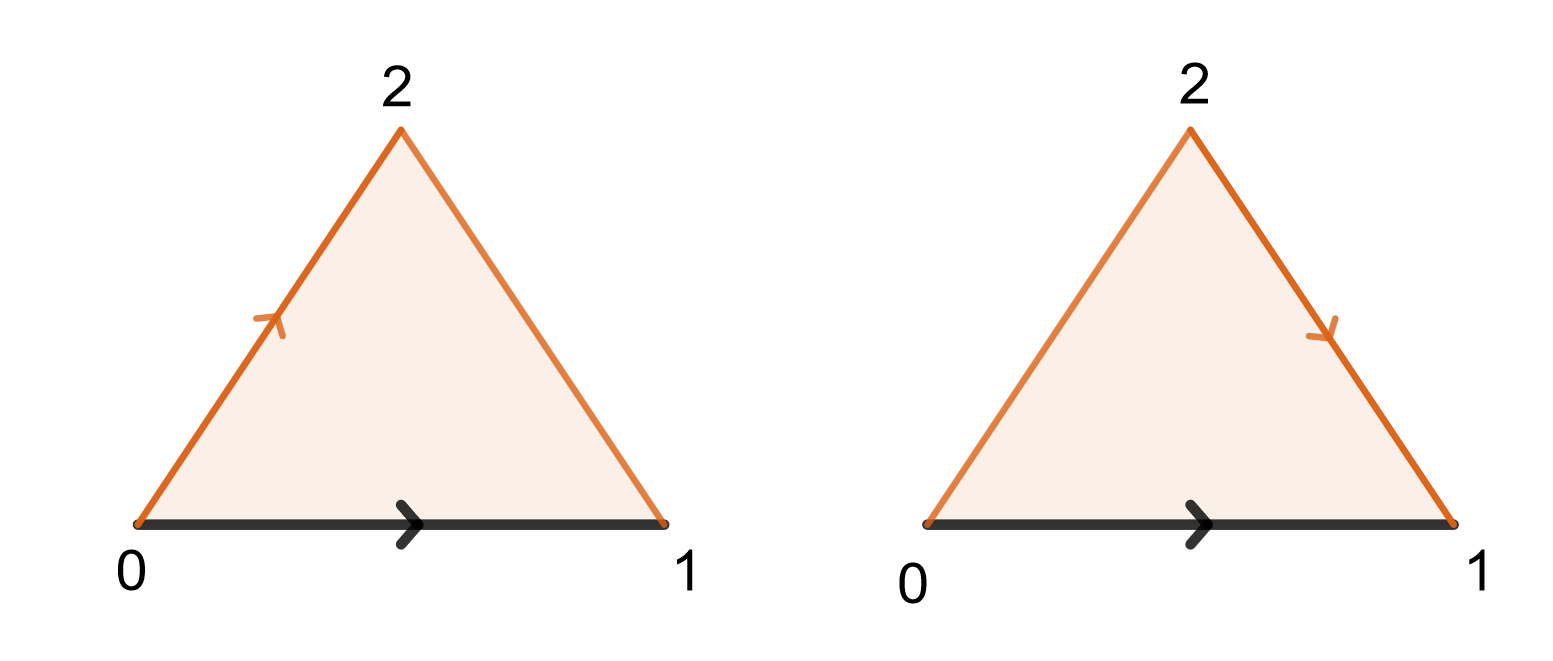}
		\caption{}
	\end{subfigure}
	\caption{Figure (a) shows the orientation induced by the positive orientation of $[0,1,2]$ on its boundary. Figure (b) shows the 1-cell $[0,1]$ along with oriented $1$-cells $\sigma$ such that $[0,1] \stackrel{X}{\sim} \sigma$ }\label{fig:exa orientation}
\end{figure}
For a $j$-cell $\sigma$, we define the boundary of $\sigma$ as $\partial \sigma= \{ \omega \subset \sigma : \text{dim}(\omega)=j-1\}$. An oriented $j$-cell $\sigma=\left[\sigma^{0}, \ldots, \sigma^{j}\right]$ induces an orientation on $\partial \sigma$ as follows: the cell $\left\{\sigma^{0}, \ldots, \sigma^{i-1}, \sigma^{i+1}, \ldots, \sigma^{j}\right\}$ is oriented as $(-1)^i\left[\sigma^{0}, \ldots, \sigma^{i-1}, \sigma^{i+1}, \ldots, \sigma^{j}\right]$, where $-\omega=\bar{\omega}$.


We use the following neighbouring relation for oriented cells introduced in \cite{ori_Spectrum_homology}. For a random simplicial complex $X$ of dimension $d$ and $\sigma, \sigma^\prime \in K_{\pm}^{d-1}$, we denote $\sigma \stackrel{X}{\sim} \sigma^{\prime}$ if there exists an oriented $d$-cell $\tau \in X$ such that both $\sigma$ and $\bar{\sigma^{\prime}}$ are in the boundary of $\tau$ as oriented cells. See Figure \ref{fig:exa orientation} (b) for an example.

\begin{definition}\label{defn:adjacency matrix}
	For a simplicial complex $X$ of dimension $d \geq 2$ with a complete $(d-1)$-skeleton, define
	\begin{enumerate}[(i)]
		\item Unsigned adjacency matrix: As the ${n \choose d} \times {n \choose d}$ matrix indexed by elements of $X^{d-1}$ with
		\begin{equation*}
			A_{\sigma, \sigma^\prime}= 
			\begin{cases}
				&1 \quad\text{if} \,\,\sigma \cup \sigma^\prime \in X, \\
				&0 \quad \text{otherwise.}
			\end{cases}
		\end{equation*}
		\item Signed adjacency matrix: As the ${n \choose d} \times {n \choose d}$ matrix indexed by elements of $X^{d-1}_+$ with
		\begin{equation*}
			A^+_{\sigma, \sigma^\prime}= 
			\begin{cases}
				&1 \quad\text{if} \,\,\sigma \stackrel{X}{\sim} \sigma^\prime , \\
				-&1 \quad \text{if}\,\, \sigma \stackrel{X}{\sim} \bar{\sigma^\prime} , \\
				&0 \quad \text{otherwise.}
			\end{cases}
		\end{equation*}
	\end{enumerate}
\end{definition}

The definition used here is rather simplistic. Nevertheless it serves our purpose as we are interested in studying the adjacency matrix as a random matrix.  For equivalent definitions of adjacency matrices and their connections with the Laplacian operator, see \cite{antti_ron},\cite{ori_ron} or \cite{uli_wagner}.

Observe that for all $j \geq 1$, there exists a bijection between $X^{d-1}$ and $X_+^{d-1}$. Therefore the entries of $A^+$ can be indexed by elements of $X^{d-1}$. Hereafter, we consider $A^+$ as a matrix indexed by elements of $X^{d-1}$. For the rest of the paper, $A_n^+$ and $A_n$ would denote the signed adjacency and unsigned adjacency matrices of the random simplicial complex $Y_d(n,p(n))$, respectively. Our first result, establishes the existence of limiting spectral distribution of adjacency matrices of the Linial-Meshulam complex when $np \rightarrow \lambda>0$.

\begin{theorem}\label{thm: existence_LSD}
	For $d \geq 2$ and $n \in \N$, let $F_{A_n}$ denote the empirical spectral distribution of unsigned adjacency matrices of $Y_d(n,p)$ and $F_{A_n^+}$ denote the empirical spectral distribution of signed adjacency matrix of $Y_d(n,p)$. Then there exists a unique probability distribution $\Gamma_d(\lambda) $ such that as $np \rightarrow \lambda >0$,
	\begin{enumerate}[(i)]
		\item $F_{A_n} \stackrel{D}{\rightarrow} \Gamma_d(\lambda)$ almost surely,
		\item  $F_{A_n^+}\stackrel{D}{\rightarrow}-\Gamma_d(\lambda)$ almost surely,
	\end{enumerate} 
	where $-\Gamma_d(\lambda)(x):=1-\Gamma_d(\lambda)(-x\_)$ for all $x \in \mathbb{R}$ and $\Gamma_d(\lambda)(y\_)$ is the left limit $\lim_{y_n\uparrow y}\Gamma_d(\lambda)(y_n)$.
\end{theorem}

We also establish several properties of the limiting spectral distribution of the adjacency matrices.

\begin{proposition}\label{prop:Gamma unbounded}
	For all $\lambda>0$, the distribution $\Gamma_d(\lambda)$ is unbounded and is not symmetric around zero.
\end{proposition}
\begin{proposition}\label{prop:Gamma_semicircle}
	As $\lambda \rightarrow \infty$, the distribution $\frac{1}{\sqrt{\lambda d}}\Gamma_d(\lambda)$ converge in distribution to the standard semi-circle law, where $\frac{1}{\sqrt{\lambda d}}\Gamma_d(\lambda)(x):= \Gamma_d(\lambda)\left(\sqrt{\lambda d}x\right)$ for all $x \in \mathbb{R}$.
\end{proposition}
In \cite{antti_ron}, it was shown that  the operator norm of $B_n^+$ under the normalization $(np(1-p))^{-1/2}$ converge almost surely to $\sqrt{d}$ as $np(1-p)$ converges to infinity. The following proposition gives the convergence of Frobenius norm for centred adjacency matrices of Linial-Meshulam complex when $np\rightarrow \lambda >0$. We use the notation $||A||_F$ to denote the Frobenius norm of a matrix $A$.
\begin{proposition}\label{prop:frobenius norm}
	For $d \geq 2$,$\lambda>0$ and $np \rightarrow \lambda$, the Frobenius norms $||B_n||_F$ and $||B_n^+||_F$ converge to $\sqrt{d\lambda}$ almost surely.
\end{proposition}

The other major results in this paper deal with the line graph of the unsigned adjacency matrix of $Y_d(n,p)$. We deal with these in the following subsection.
\subsection{Line graph of $Y_d(n,p)$}\label{subsec:line grap} 
In this subsection, we introduce the necessary preliminaries on local weak limit and line graph, and state our main results. First, we define the space of unlabelled rooted graphs. 

A rooted graph is a pair $(G,o)$ where $G$ is a graph with a distinguished vertex $o$, called the root. Two rooted graphs $(G_1,o_1)$ and $(G_2,o_2)$ with $G_1=(V_1,E_1)$ and $G_2=(V_2,E_2)$ are said to be isomorphic if there exists a bijection $\sigma : V_1 \rightarrow V_2$ such that $\sigma(o_1)=o_2$ and $\sigma(G_1)=G_2$ where $\sigma$ acts on $E_1$ as $\sigma(\{u,v\})=\{\sigma(u),\sigma(v)\}$. We denote the isomorphism between $(G_1,o_1)$ and $(G_2,o_2)$ as $(G_1,o_1) \simeq (G_2,o_2)$.

A graph $G=(V,E)$ is called locally finite if $\operatorname{deg}_G(v) < \infty$ for all $v \in V$. Clearly, every finite graph is locally finite. In this section, unless stated otherwise, every graph considered is locally finite. We also maintain the convention that on considering a sequence of graphs $(G_i)_{i \in \mathbb{N}}$, we would use $V_i$ and $E_i$ to denote the vertex set and edge set of $G_i$, respectively. 

Isomorphism of rooted graphs is an equivalence relation on the class of all rooted locally finite connected graphs and an equivalence class under this relation is called an unlabelled rooted graph. We denote the space of all unlabelled rooted graphs by $\mathcal{G}^*$. 

For a connected graph $G$ and vertices $u,v \in V(G)$, the graph distance between $u$ and $v$, $\operatorname{dist}(u,v)$, is defined as the number of edges in a shortest path between $u$ and $v$. For a  rooted graph $(G,o)$ and $t \geq 0$, we consider the vertex set 
$V_t^\prime=\{v \in V(G): \operatorname{dist}(o,v)\leq t\}$. The subgraph induced by $(G,o)$ on $V_t^\prime$ is denoted by $(G,o)_t$. Note that if $(G_1,o_1)\simeq (G_2,o_2)$, then $(G_1,o_1)_t\simeq (G_2,o_2)_t$ for all $t \geq 0$, and therefore for $g \in \mathcal{G^*}$, the unlabelled graph $(g)_t$ is well-defined as the equivalence class of $(G,o)_t$ for some $(G,o)$ belonging to the equivalence class $g$. 

It is known that $\mathcal{G}^*$ is a metric space with the metric
\begin{equation}\label{eqn: metric G*}
	d_{\mathcal{G}^*}(g,h)= \frac{1}{1+T_{g,h}} \quad \text{where} \quad T_{g,h}= \sup\{t \in \mathbb{N}: (g)_t \simeq (h)_t\}.
\end{equation}
It follows from (\ref{eqn: metric G*}) that the distance between any two unlabelled graphs is bounded above by 1 and for $\epsilon >0$ and $g \in \mathcal{G^*}$, the open ball $B_\epsilon(g)$ is given by
\begin{align*}
	B_\epsilon(g)&=\mathcal{G^*} \text{ if }\epsilon \geq 1 \text{ and }\\
	B_\epsilon(g)&=\{h:(h)_{\lceil\frac{1}{\epsilon}\rceil} \simeq (g)_{\lceil\frac{1}{\epsilon}\rceil}\} \text{ for } \epsilon < 1.
\end{align*}

The resulting metric topology is known as the local topology on $\mathcal{G^*}$ and clearly, the local topology on $\mathcal{G}^*$ is the smallest topology such that for any $g \in \mathcal{G}^*$ and any integer $t \geq 1$, the map $f\left((G,o)\right)= \mathbbm{1}_{(G,o)_t \simeq g}$ is continuous. Under local topology, $\mathcal{G}^*$ is complete and separable. For proofs of these facts, see  \cite{coursSRG} or \cite{remco_van}.
\begin{definition}
	For a finite graph $G=(V,E)$, let $G(v)$ denote the connected component of $G$ containing $v \in V$. We define the probability measure $U(G)$ on $\mathcal{G}^*$ as the law of equivalence class of the rooted graph $\left( G(o),o \right)$ where $o$ is chosen uniformly from $V$. 
\end{definition}
In other words, for a finite graph $G$, $U(G)$ is given by
\begin{equation}\label{eqn: U(G)}
	U(G)=\frac{1}{|V|} \sum_{o \in V} \delta_{[(G(o),o)]},
\end{equation}
where $[(G(o),o)]$ is the equivalence class of $\left(G(o),o\right)$ in $\mathcal{G}^*$ and $\delta_{[(G(o),o)]}$ denote the Dirac measure concentrated at $[(G(o),o)]$.

We denote the set of all probability measures on $\mathcal{G}^*$ with $\mathcal{P}(\mathcal{G}^*)$. For $\mu_n, \mu \in \mathcal{P}(\mathcal{G}^*) $, we say $\mu_n$ converges to $\mu$ weakly if 
\begin{equation*}
	\int_{\mathcal{G}^*}h\left( G,o \right) d\mu_n \rightarrow \int_{\mathcal{G}^*}h\left( G,o \right) d\mu,
\end{equation*}
for all bounded continuous functions $h$ on $\mathcal{G}^*$.
\begin{definition}
	Let $(G_n)_{n \geq 1}$ be a sequence of finite graphs, we say that the probability measure $\rho \in \mathcal{P}(\mathcal{G}^*)$ is the local weak limit of $G_n$ if $U(G_n)$ converges weakly to $\rho$.
\end{definition}
We shall use $A_n$ to denote the unsigned adjacency matrix of the Linial-Meshulam complex   $Y_d(n,p)$. We define $U_n= K^d(n)$ and $V_n=K^{d-1}(n)$, where $K^j(n)$ is the set of all $j$-cells on $n$ vertices.

First, we define a graph with same adjacency matrix as the unsigned adjacency matrix of the Linial-Meshulam complex $Y_d(n,p)$. This graph is called the line graph of the simplicial complex and is a subgraph of the intersection graph on the family of all $(d-1)$ cells.
\begin{definition}\label{defn: graph Gn}
	For $d \geq 2, n \geq d+1$ and $0 < p < 1$,  consider the random simplicial complex $Y_d(n,p)$ with unsigned adjacency matrix $A_n$. The line graph of $Y_d(n,p)$ is defined as the graph $G_n =(V_n, E_n)$ where $V_n=K^{d-1}(n)$ and for $v, v^{\prime} \in V_n$, $\{v,v^{\prime}\} \in E_n$ if $(A_n)_{vv^{\prime}}=1$. Notice that adjacency matrix of the line graph of $Y_d(n,p)$ is $A_n$, the unsigned adjacency matrix of $Y_d(n,p)$. 
\end{definition}

Next, we introduce a special type of random graph known as $d$-block Galton-Watson graph. A Galton-Watson process with offspring distribution $P$ is defined as follows: Let $\{X_{i,n}\}$ be a sequence of i.i.d random variables with distribution, $P$. Define
\begin{align*}
	Z_0 &\equiv 1 \text{ and}\\
	Z_{n+1} &=\sum_{i=1}^{Z_n} X_{i,n} \text{ for all } n \geq 0.
\end{align*}
 The process $\{Z_n\}_{n \geq 1}$ is called a Galton-Watson process. Here $Z_n$ denotes the number of individuals at the $n$-th generation and $X_{i,n}$ represents the number of offsprings of $i$. For a vertex $i$ of $n$-th generation, we use the notation $(i,1),(i,2),\ldots,(i,X_{i,n})$ to denote the offsprings of $i$. A Galton-Watson tree with offspring distribution $P$, is constructed by considering a Galton-Watson process process $Z_n$ with offspring distribution $P$, and then joining all vertices to their offsprings via edges.
For $d \in \N$, a $d$-block Galton-Watson graph with offspring distribution $dP$ is constructed by taking $X_{i,n} \sim dP$ and constructing the following edges:
\begin{enumerate}[(i)]
	\item Edges exist between vertices and their offsprings.
	\item There exist an edge between $(i,j_1)$ and $(i,j_2)$, if $\lfloor \frac{j_1-1}{d} \rfloor=\lfloor \frac{j_2-1}{d} \rfloor$.
	\item No other edges exist in the graph.
\end{enumerate}
$d$-block Galton-Watson graph will be defined in more concrete fashion in Section \ref{sec: local_weak}. For a visual representation of $d$-block Galton-Watson graph see Figure \ref{fig: d-block graph}. Our next main theorem gives the local weak limit of the line graph of $Y_d(n,p)$.
\begin{figure}
	\captionsetup[subfigure]{justification=centering}
	\hspace{-10mm}
	\begin{subfigure}[b]{0.45\textwidth}
		\centering
		\includegraphics[height=45mm, width=\textwidth]{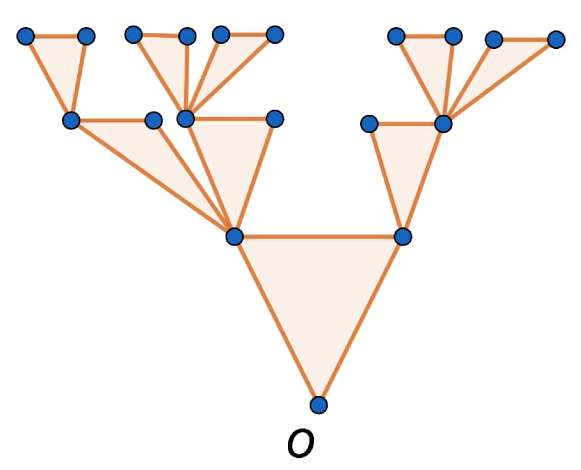}
		\caption{}
	\end{subfigure}
	\hspace{10mm}
	\begin{subfigure}[b]{0.45\textwidth}
		\centering
		\includegraphics[height=45mm, width=\textwidth]{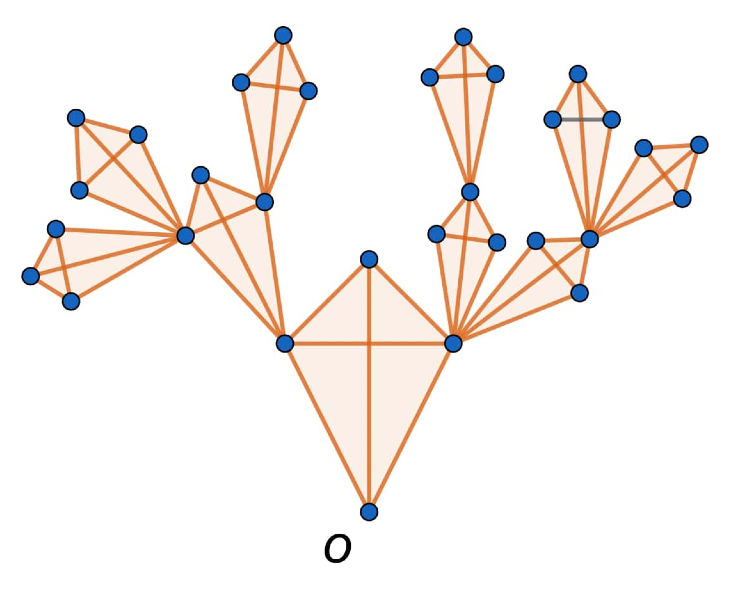}
		\caption{}
	\end{subfigure}
	\caption{(a) shows a finite 2-block Galton-Watson graph and (b) shows a finite 3-block Galton-Watson graph.}	\label{fig: d-block graph}
\end{figure}

\begin{theorem}\label{thm: weak limit}
	For $d \geq 2$ and $\lambda> 0$, let $G_n$ be the line graph of $Y_d(n,\lambda/n)$. Then the local weak limit of $G_n$ is $dGW(d\operatorname{Poi}(\lambda))$, the probability measure on $\mathcal{G}^*$ induced by $d$-block Galton-Watson graph with offspring distribution $d \operatorname{Poi}(\lambda)$.
\end{theorem}

Figure \ref{fig:limiting_dist.jpg} suggests that for small values of $\lambda$, $\Gamma_d(\lambda)$ does not have a continuous part. In the next theorem, we gave a necessary condition for the limiting spectral distribution of adjacency matrices of Linial-Meshulam complex to have a continuous part. Our work here is inspired by \cite{BVS}, in which the authors have showed that for $np \rightarrow \lambda$, the limiting spectral distribution of the adjacency matrix of Erd\H{o}s-R{\'e}nyi random graph has a continuous part if and only if $\lambda>1$.
\begin{theorem}\label{thm: spectral measure dGW}
For $\lambda \leq 1/d$, the distribution $\Gamma_d(\lambda)$ is purely atomic.
\end{theorem}
\section{LSD of adjacency matrices of Linial-Meshulam complex}\label{sec: Existence LSD}

Our central objective in this section is to prove Theorem \ref{thm: existence_LSD}. In Section \ref{subsec: LSD of centred adjacency matrix}, we prove that the centred adjacency matrices of $Y_d(n,p)$ have the same LSD as in Theorem \ref{thm: existence_LSD} (see Theorem \ref{thm: centred LSD}). In Section \ref{subsec:LSD for adjacency matrix}, it is shown that the eigenvalues of adjacency matrices are sufficiently close to the eigenvalues of their centred version, and therefore have the same LSD. We end this section with a few observations on the LSD obtained using its moments. We begin by defining the necessary combinatorial objects and the preliminary results required in the proof of Theorem \ref{thm: existence_LSD}.
\subsection{Combinatotial objects for Linial-Meshulam complex}

Recall the definition of a cell from Definition \ref{defn: d-cell}. For $j \geq -1$, we denote the set of all $(d-1)$ cells on $\{1,2,\ldots,n\}$ by $K^j$. Here, the dependence on $n$ is suppressed for notational convenience.
\begin{definition}\label{defn: word}
We define an element of $K^{d-1}$ as a letter. A word of length $k \geq 1$ is a sequence $\sigma_{1} \sigma_{2} \cdots \sigma_{k}$ of letters such that $\sigma_{i} \cup \sigma_{i+1}$ is a $d$-cell. For a word $w=\sigma_{1} \sigma_{2} \cdots \sigma_{k}$, we define $\operatorname{supp}_{0}(w)=\sigma_{1} \cup \sigma_{2} \cup \ldots \cup \sigma_{k}$ and $\operatorname{supp}_{d}(w)=\left\{\sigma_{i} \cup \sigma_{i+1} : 1 \leq i \leq k-1\right\}.$ The set $\left\{\sigma_{i}, \sigma_{1+1}\right\}$ is called an edge. We denote by $N_{w}(e)$, the number of times an edge $e$ is crossed by $w$. For  $\tau \in K^d$, we define $N_{w}(\tau):=\sum_{e \in E_{w}(\tau)} N_{w}(e)$, where $E_{w}(\tau)=\left\{\left\{\sigma_{i}, \sigma_{i+1}\right\}: \sigma_{i} \cup \sigma_{i+1}=\tau\right\}$.
\end{definition}

Consider the word $w=\{1,2\}\{1,3\}\{3,4\}\{1,4\}\{2,4\}\{1,2\}$. For the word $w$, we have $\operatorname{supp}_0(w)=\{1,2,3,4\}$ and $\operatorname{supp}_d(w)=\{\{1,2,3\},\{1,3,4\},\{1,2,4\}\}$. Here note that each edge appears only once and therefore $N_w(e)=1$ for all edges $e$. Now, for $d$-cells $\tau \in \operatorname{supp}_d(w)$, $N_w(\tau)$ denotes the number of $1 \leq i \leq 5$ such that $\sigma_i \cup \sigma_{i+1}=\tau$, and therefore $N_w(\{1,2,3\})=1$, $N_w(\{1,3,4\})=2$, $N_w(\{1,2,4\})=2$.
\begin{definition}\label{defn:equi_word}
	We say two words $w=\sigma_{1} \sigma_{2} \cdots \sigma_{k}$ and $w^\prime=\sigma_{1}^\prime \sigma_{2}^\prime \cdots \sigma_{k}^\prime$ are equivalent if there exists a bijection $\pi:\operatorname{supp}_0(w) \rightarrow \operatorname{supp}_0(w^\prime)$ such that $\pi(\sigma_i)=\sigma_i^\prime$ for all $i$ and $\pi|_{\sigma_1}$ is a strictly increasing function. 
\end{definition}

Consider the words $w=\{1,2\}\{1,3\}\{3,4\}\{2,4\}\{1,2\}$ and $w^\prime=\{3,5\}\{1,3\}\{1,7\}\{5,7\}\{3,5\}$. Here $\operatorname{supp}_0(w)=\{1,2,3,4\}$ and $\operatorname{supp}_0(w^\prime)=\{3,5,1,7\}$. The words $w$ and $w^\prime$ are equivalent and the appropriate bijection $\pi$ is given by $\pi(1)=3,\pi(2)=5,\pi(3)=1,\pi(4)=7$ and this is the only possible bijection showing equivalence between $w$ and $w^\prime$.

\begin{definition}\label{defn: W_s^k}
	A word $ w=\sigma_{1} \sigma_{2} \cdots \sigma_{k}$ is said to be closed if $\sigma_{k}=\sigma_{1}$. We denote by $\mathcal{W}_s^k(n,d)$, the set of representatives for the equivalence class of closed words of length $k+1$ and $|\operatorname{supp}_0(w)|=s$ such that $N_w(\tau) \neq 1$ for every $\tau \in K^d$. When $n,d$ are clear from the context, we denote the set $\mathcal{W}_s^k(n,d)$ by $\mathcal{W}_s^k$.
\end{definition}
\begin{figure}[h!]
	\includegraphics[height=35mm, width =105mm]{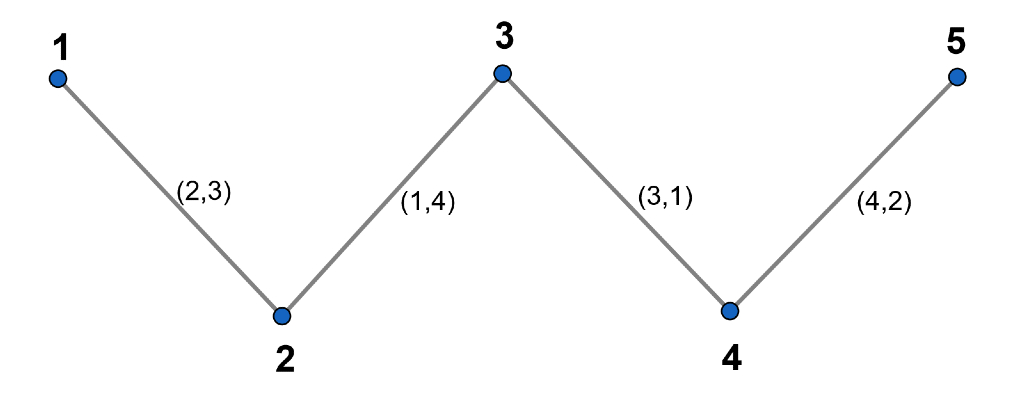}
	\caption{Consider the word $w=\{2,3\}\{1,2\}\{1,4\}\{2,4\}\{2,3\}$. Above figure shows the labelled path corresponding to $w$. The label $r_i$ is determined by the element in $\sigma_{i}\setminus \left(\sigma_{i+1}\cap \sigma_{i}\right)$ and $s_i$ is determined by the element in $\sigma_{i+1}\setminus \left(\sigma_{i+1}\cap \sigma_{i}\right)$. The ordering $\eta$ is given by $\eta(\{2\})=1,\eta(\{3\})=2,\eta(\{1\})=3,\eta(\{4\})=4$. Thus $r_1=2,r_2=2,r_3=1$ and $r_4=2$ and $s_1=3,s_2=4,s_3=1$ and $s_4=2$.
	}
	\label{fig:example_word}	
\end{figure}
Consider a word $w=\sigma_{1} \sigma_{2} \cdots \sigma_{k}$ with vertex set $V=[n]$. We define a new ordering $\eta:\operatorname{supp}_0(w) \rightarrow
\{1,2,\ldots, |\operatorname{supp}_0(w)|\}$ in the following way: For $u,v \in \operatorname{supp}_0(w)$, we say $\eta(u)<\eta(v)$ if 
\begin{enumerate}[(a)]
	\item $\min\{j:u \in \sigma_j\}<\min\{j:v \in \sigma_j\}$ and
	\item for $\min\{j:u \in \sigma_j\}=\min\{j:v \in \sigma_j\}$, if $u <v$ as natural numbers.
\end{enumerate}
We remark here that condition (b) occurs only when $\min\{j:u \in \sigma_j\}=1$.

For the word $w$, we construct a rooted labelled graph $G_w= (V_w,E_w,\ell_w,o)$ where the vertex set $V_w=[k]$, the edge set is $E_w=\{\{i,i+1\}:1\leq i \leq k-1\}$ and the root $o=1$. Note that since $\sigma_{i}\cup \sigma_{i+1}$ is always $d$-cell, both $\sigma_{i}\setminus \sigma_{i+1}$ and $\sigma_{i+1}\setminus \sigma_i$ are singleton sets. Now we define the labelling $\ell_w$ as follows: for an edge $\{i,i+1\} \in E_w$, we define
\begin{equation*}
	\ell_w\left(\{i,i+1\}\right)=(r_i,s_i),
\end{equation*} 
where $r_i=\eta(x_i)$, with $\{x_i\}=\sigma_i \setminus \sigma_{i+1}$ and $s_i=\eta(y_i)$, with $\{y_i\}=\sigma_{i+1}\setminus \sigma_i$.
For an example, see the construction of rooted labelled graph in Figure \ref{fig:example_word}. It follows that two words $w$ and $w^\prime$ are said to equivalent if and only if $\ell_w\left((i,i+1)\right)=\ell_{w^\prime}\left((i,i+1)\right)$ for all $1 \leq i \leq k-1$. The notion of labelled graph reduces the number of possible elements $w \in \widetilde{\mathcal{W}}_s^k$ to the order of $O(s^k)$. This shall be used to compute values of $|\widetilde{\mathcal{W}}_s^k|$ in Table \ref{table:W_s^k}.

Consider a word $w$ such that $|\operatorname{supp}_0(w)|=s$. We use the following convention to denote a representative element of $[w]$. The first $(d-1)$-cell is always taken as $\{1,2,\ldots ,d\}$ and the subsequent new $0-$cells appear in ascending order from $\{1,2,\ldots ,s\}$. For example, the representative element for the word $\{3,4\}\{3,6\}\{6,5\}\{3,6\}\{3,4\}$ would be $\{1,2\}\{1,3\}\{3,4\}\{1,3\}\{1,2\}$.

Next, we prove some preliminary lemmas about words required in the proof of Theorem \ref{thm: existence_LSD}. The reader may skip the proofs for the time being and read them together with the proof of Theorem \ref{thm: existence_LSD} in Section \ref{subsec:LSD for adjacency matrix}. 
\begin{lemma}\label{claim: supp_0,supp_d}
	For every word $w$, $\left|\operatorname{supp}_{0}(w)\right| \leq\left|\operatorname{supp}_{d}(w)\right|+d$.
\end{lemma}
\begin{proof}
	We prove the result by induction on $k$, the length of the word. For $k=1$, $\left|\operatorname{supp}_{0}(w)\right|=d$, $\left|\operatorname{supp}_{d}(w)\right|=0$ and therefore the inequality holds. Suppose the inequality holds for all words of length $k$, for some fixed $k \geq 1$. Consider a word $w=\sigma_{1} \sigma_{2} \cdots \sigma_{k+1}$ and let $v= \sigma_{1} \sigma_{2} \cdots \sigma_{k}$. Then by induction hypothesis, $\left|\operatorname{supp}_{0}(v)\right| \leq \left|\operatorname{supp}_{d}(v)\right| +d $. The following two cases occur:\\\\
	\textbf{Case 1:} $\left|\operatorname{supp}_{0}(w)\right| = \left|\operatorname{supp}_{0}(v)\right|$.\\\\ 
	Since, $\left|\operatorname{supp}_{d}(v)\right| \leq \left|\operatorname{supp}_{d}(w)\right|$, it follows from induction hypothesis that \begin{equation*}
		\left|\operatorname{supp}_{0}(w)\right| =\left|\operatorname{supp}_{0}(v)\right| \leq \left|\operatorname{supp}_{d}(v)\right|+d \leq \left|\operatorname{supp}_{d}(w)\right|+d.
	\end{equation*}
	\\
	\noindent\textbf{Case 2:} $\left|\operatorname{supp}_{0}(w)\right| >  \left|\operatorname{supp}_{0}(v)\right|$.\\\\
	Since $\sigma_k \cup \sigma_{k+1}$ is a $d$-cell, $\sigma_{k+1}$ can contain at most one new $0$-cell and thus, the only possible value of $\left|\operatorname{supp}_{0}(w)\right|$ in this case is $(\left|\operatorname{supp}_{0}(v)\right|+1)$. As $\sigma_{k+1}$ contains a $0$-cell not contained in $\left|\operatorname{supp}_{0}(v)\right|$, $\sigma_k \cup \sigma_{k+1} \notin \operatorname{supp}_{d}(v)$ and thus $|\operatorname{supp}_d(w)|=|\operatorname{supp}_d(v)|+1$. Hence we get,
	\begin{equation*}
		\left|\operatorname{supp}_{0}(w)\right| =\left|\operatorname{supp}_{0}(v)\right|+1 \leq \left|\operatorname{supp}_{d}(v)\right|+d+1 = \left|\operatorname{supp}_{d}(w)\right|+d.
	\end{equation*}
	
	This completes the proof.
\end{proof}

Recall the definition of $\mathcal{W}_s^k$ from Definition \ref{defn: W_s^k}. Also note that for $k=1$, the set $\mathcal{W}_s^1$ is an empty set. The following claim gives an upper bound on the size of $\mathcal{W}_s^k$ for $k \geq 2$.

\begin{lemma}\label{claim:size W_k,s}
	For every $d \geq 2$, $k \geq 2$ and $ d+1 \leq s \leq \bigl\lfloor\frac{k}{2}\bigr\rfloor+d$,
	$$|\mathcal{W}_s^k| \leq  (ds)^{k}.$$
\end{lemma}

\begin{proof}
	Notice that $\mathcal{W}_s^k$ is non-empty only if $d+1 \leq s \leq \bigl\lfloor\frac{k}{2}\bigr\rfloor+d$. We prove the claim by enumerating the number of different possible representatives for elements of $\mathcal{W}_s^k$. Note that $\sigma_1$ is always given by $\{1,2, \ldots ,d\}$ and therefore $\sigma_{1}$ has only one possibility. For each $1 \leq i \leq k-1$, $\sigma_{i+1} \setminus \sigma_{i}$ is a singleton set. Therefore $\sigma_{i+1}$ can be obtained from $\sigma_{i}$ by the deletion of a $0-$cell from $\sigma_{i}$ followed by the addition of a new $0-$cell. Therefore for each fixed choice of $\sigma_{i}$, the number of choices for $\sigma_{i+1}$ is bounded by $d \times s$. It follows that
	$$|\mathcal{W}_s^k| \leq d \times (ds)^{k-1} \leq  (ds)^k.$$
\end{proof}

The following notation is inspired by Lemma \ref{claim: supp_0,supp_d}. 

\begin{notation}\label{notation:W_s^k tilde}
	We denote by $\widetilde{\mathcal{W}}_s^k$, the set of all closed words $w \in \mathcal{W}_s^k$ such that $|\operatorname{supp}_d(w)|=s-d$.
\end{notation}

\begin{definition}\label{defn:subword}
	For a word $w=\sigma_1\sigma_{2}\cdots \sigma_{k+1} \in \widetilde{\mathcal{W}}_s^k$,  $\tilde{w}=\sigma_{i_1}\sigma_{i_2}\cdots \sigma_{i_\ell}$ such that $1 \leq i_1 <i_2<\cdots < i_\ell \leq k+1$ are consecutive integers, is called a subword of $w$. For the word $w$ and $1 \leq i \leq k+1$, define the subword $\tilde{w}_i$ as $\tilde{w}_i=\sigma_{1}\sigma_{2}\cdots \sigma_{i}$.
\end{definition}
We now briefly explain a special characteristic of words $w \in \widetilde{\mathcal{W}}_s^k$. 
Note that, $|\operatorname{supp}_0(\tilde{w}_1)|=d$ and $|\operatorname{supp}_d(\tilde{w}_1)|=0$, and thus $|\operatorname{supp}_d(\tilde{w}_1)|=|\operatorname{supp}_0(\tilde{w}_1)|-d$. Also, we have that $|\operatorname{supp}_0(\tilde{w}_i)|$ and $|\operatorname{supp}_d(\tilde{w}_i)|$ are increasing functions on $i$. Furthermore, $|\operatorname{supp}_0(\tilde{w}_{i+1})|-|\operatorname{supp}_0(\tilde{w}_i)|$ and $|\operatorname{supp}_d(\tilde{w}_{i+1})|-|\operatorname{supp}_d(\tilde{w}_i)|$ are always bounded above by 1. Thus we get that 
\begin{equation}\label{eqn:iff W_s^k}
	|\operatorname{supp}_d(w)|=s-d \text{ if and only if } |\operatorname{supp}_d(\tilde{w}_{i})|=|\operatorname{supp}_0(\tilde{w}_i)|-d \text{ for all } 1 \leq i \leq k+1,
\end{equation}
or in other words, $\operatorname{supp}_0(\tilde{w}_{i+1})$ contains a new $0$-cell not contained in $\operatorname{supp}_0(\tilde{w}_i)$ if and only if $\operatorname{supp}_d(\tilde{w}_{i+1})$ contains a new $d$-cell not contained in $\operatorname{supp}_d(\tilde{w}_i)$. Now, we prove the following claim about the number of elements in each equivalence class in $\widetilde{\mathcal{W}}_s^k$.
\begin{lemma}\label{claim:card_[w]}
	For every $d \geq 2, k \geq 1$ and $w \in \widetilde{\mathcal{W}}_s^k$, $|[w]|=\frac{n!}{(n-s)!d!}$.
\end{lemma}
\begin{proof}
	Consider a word $w \in \widetilde{\mathcal{W}}_s^k$. Note that the number of choices for choosing $s$ many 0-cells in $\supp_0(w)$ is ${n \choose s}$. Our objective is to enumerate the number of choices for $\sigma_i$ for each fixed choice of $\operatorname{supp}_0(w)$. For a fixed choice of $\operatorname{supp}_0(w)$, the number of ways for choosing $\sigma_{1}$ is ${s \choose d}$. 
	
	Suppose $\sigma_1,\sigma_2, \ldots ,\sigma_i$ are chosen. Consider the graph $G_w$ corresponding to $w$ with labelling $\ell_w$ and let $\ell(\{i,i+1\})=(r_i,s_i)$. Suppose there exits $j < i$ such that $s_j=s_i$. This implies that the $d$-cell $\sigma_i \cup \sigma_{i+1}$ has already occurred in $w$. Thus, all the $0$-cells of $\sigma_{i+1}$ are already chosen. Furthermore, $\sigma_{i+1} \setminus \sigma_{i}$ is a singleton set and therefore, there exists a unique choice for the $0$-cell to be added. Thus in this case, $\sigma_{i+1}$ has only one choice.
	
	Suppose $i$ is the first occurrence of $s_i$ as the second component. This implies that $\sigma_i \cup \sigma_{i+1}$ is the first occurrence of a $d$-cell. Since $w \in \widetilde{\mathcal{W}}_s^k$, it follows that $\sigma_{i+1}$ contains a $0$-cell that is not an element of $\supp_0(\tilde{w}_i)$ where $\tilde{w}_i=\sigma_{1}\sigma_{2}\cdots \sigma_{i}$. If $\sigma_i \cup \sigma_{i+1}$ is the $r$-th $d$-cell appearing in the word $w$, then the number of ways of choosing the new $0$-cell, and thereby $\sigma_{i+1}$, is ${s-d-r \choose 1}$. Note that $|\supp_d(w)|=s-d$ and thus the number of choices for $|[w]|$ is
	\begin{align*}
		|[w]|&={n \choose s}{s \choose d}\prod_{r=0}^{s-d-1}{s-d-r \choose 1}\\
		&=\frac{n!}{(n-s)!d!}.
	\end{align*}
\end{proof}

\begin{definition}
	For $\lambda>0 $ and $k \in \Z^+$, define 
	$$\displaystyle\beta_{k}(\lambda)=\sum_{s=d+1}^{\lfloor\frac{k}{2}\rfloor +d}\left|\widetilde{W}_{s}^{k}\right| \lambda^{s-d},$$
	where $\widetilde{\mathcal{W}}_s^k$ is as defined in Notation \ref{notation:W_s^k tilde}. Note that $\beta_1(\lambda)=0$ for all $\lambda>0$.
	
\end{definition}
\begin{table}[h!]
	\begin{tabular}[]{c c}
		\setlength\extrarowheight{4pt}
		\renewcommand{\arraystretch}{1.3}
		\hspace{-10mm}\begin{minipage}{.5\linewidth}
			\begin{tabular}[pos]{|c|c|c|c|c|c|c|c|c|}
				\hline
				$d=2$ & \multicolumn{8}{c}{Values of $k$} \vline \\
				\hline
				values of $s$ & \textbf{1} & \textbf{2} &  \textbf{3} & \textbf{4} & \textbf{5} &  \textbf{6} & \textbf{7} & \textbf{8} \\
				\hline
				3& 0 & 2 & 2 &  6 & 10 & 22 &  42 & 86 \\ 
				4& 0 & 0 & 0 &  8 & 20 & 84 & 224 & 688\\
				5& 0 & 0 & 0 &  0 & 0 & 40 &  168 & 896 \\
				6& 0 & 0 & 0 &  0 & 0 & 0  &  0   & 224 \\
				\hline  
			\end{tabular}
		\end{minipage}&
		\begin{minipage}{.5\linewidth}
			\setlength\extrarowheight{4pt}
			\renewcommand{\arraystretch}{1.3}
			\begin{tabular}[pos]{|c|c|c|c|c|c|c|c|c|}
				\hline
				$d=3$  & \multicolumn{8}{c}{Values of $k$} \vline  \\
				\hline
				values of $s$& \textbf{1} & \textbf{2} &  \textbf{3} & \textbf{4} & \textbf{5} &  \textbf{6} & \textbf{7} & \textbf{8} \\
				\hline
				
				4& 0 & 3 & 6 &  21 & 60 & 183 &  546 & 1641 \\
				5& 0 & 0 & 0 &  18 & 90 & 486 & 2142 & 9198\\
				6& 0 & 0 & 0 &  0 & 0 & 135 &  1134 & 8316 \\
				7& 0 & 0 & 0 &  0 & 0 & 0  &  0   & 1134 \\
				\hline  
			\end{tabular}
		\end{minipage}
	\end{tabular}
	\caption{Table showing the values of $|\widetilde{\mathcal{W}}_s^k|$ for $d=2,3$ and different values of $s,k$.}
	\label{table:W_s^k}
\end{table}
We will later show that for each $\lambda>0$, $\{\beta_{k}(\lambda)\}$ is the sequence of limiting moments of the EESD of the centred unsigned adjacency matrix of $Y_d(n,p)$ as $np \rightarrow \lambda$. Therefore $\{\beta_{k}(\lambda)\}$ is a positive semi-definite sequence and there exist probability measures on $\mathbb{R}$ with moment sequence $\{\beta_{k}(\lambda)\}$.
We proceed to show that the moment sequence $\{ \beta_k (\lambda) \}$ determine a unique probability measure on $\R$ to which $F_{B_n}$ converges almost surely.

\begin{lemma}\label{lemma: unique:beta_h}
	For every $\lambda>0$, the moment sequence $\{\beta_{k}(\lambda)\}_{k \geq 1}$ defines a unique probability measure on the real line.
\end{lemma}
\begin{proof}
	We prove the lemma using Carleman's condition (Lemma B.3, \cite{bai_silverstein}). Consider the case $\lambda>1$. Note that since $\widetilde{\mathcal{W}}_s^{2k} \subseteq \mathcal{W}_s^{2k}$, it follows from Lemma \ref{claim:size W_k,s} that $|\widetilde{\mathcal{W}}_s^{2k}| \leq (d(k+d))^{2k}$, and in turn
	\begin{align*}
		\beta_{2k}(\lambda) & \leq k \times(d(k+d))^{2k} \times \lambda^{2k}\\
		&\leq (k+d)^{2k+1} \times d^{2k} \times \lambda^{2k}.
	\end{align*}
	Thus 
	$\displaystyle {(\beta_{2k}(\lambda))}^{-\frac{1}{2k}} \geq a_k:= \frac{1}{d\lambda} \times \frac{1}{(k+d)^{1+1/2k}} $. Since $\sum_{k \geq 1} a_k$ diverges, it follows from Carleman's condition that there exists a unique probability measure on $\R$ with $\{ \beta_k(\lambda)\}$ as its moments. For $\lambda \leq 1$, we get $\beta_{2k}(\lambda) \leq k \times(d(k+d))^{2k}$ and a similar argument completes the proof. 
\end{proof}

Now, with all the required preliminaries in hand, we are equipped to prove Theorem \ref{thm: centred LSD}.

\subsection{LSD of centred adjacency matrices}\label{subsec: LSD of centred adjacency matrix}

The main objective of this section is to prove the existence of limiting spectral distribution for the centred versions of the adjacency matrices of Linial-Meshulam complex. We define the centred adjacency matrices of $Y_d(n,p)$ as
	\begin{equation*}
		B_n= A_n- \mathbb{E} A_n \text{ and } B_n^{+}= A_n^+- \mathbb{E} A_n^+.
	\end{equation*}
In this section, we show that $B_n$ and $B_n^+$, have limiting spectral distribution $\Gamma_d(\lambda)$ and $-\Gamma_d(\lambda)$, respectively. Later, this theorem will be used to prove Theorem \ref{thm: existence_LSD}.

\begin{theorem}\label{thm: centred LSD}
	For $d \geq 2$ and $n \in \N$, let $F_{B_n}$ and $F_{B_n^+}$ denote the empirical spectral distribution of $B_n$ and $B_n^+$, respectively. Then as $np \rightarrow \lambda >0$,
	\begin{enumerate}[(i)]
		\item $F_{B_n} \stackrel{D}{\rightarrow} \Gamma_d(\lambda)$ almost surely,
		\item  $F_{B_n^+}\stackrel{D}{\rightarrow}-\Gamma_d(\lambda)$ almost surely,
	\end{enumerate} 
	where $\Gamma_d(\lambda)$ is the unique distribution function on real line with moment sequence $\{\beta_k\}$.	
\end{theorem}
We prove Theorem \ref{thm: centred LSD} for the signed and unsigned cases in subsection \ref{subsubsec: Proof of centred unsigned} and subsection \ref{subsubsec: Proof of centred signed}, respectively. In both cases, we first show that the expected empirical spectral distribution converges to $\Gamma_d(\lambda)$ (or $-\Gamma_d(\lambda)$). Then, using a concentration argument, we show that the empirical spectral distribution is sufficiently close to the expected empirical spectral distribution and therefore, both have the same limit. 
 \subsubsection{Proof of Theorem \ref{thm: centred LSD}: Unsigned adjacency matrix}\label{subsubsec: Proof of centred unsigned}

 In the following theorem, we establish the convergence of expected empirical spectral distribution for the centred unsigned adjacency matrix.
\begin{theorem}\label{thm:EESD_centred_unsigned}
For $np \rightarrow \lambda>0$, $\mathbb{E}F_{B_n}$, the expected empirical spectral distribution of the centred unsigned adjacency matrix $B_n$, converges weakly to $\Gamma_d(\lambda)$, the distribution uniquely determined by the moment sequence $\beta_{k}(\lambda)$.
\end{theorem}
\begin{proof}

For $k=0$, we have $\int_{\mathbb{R}} x^k\mu_{B_n}(dx)=1$.
For $k \geq 1$, 
\begin{equation}\label{eqn:k-th moment exp1}
\mathbb{E}\left[\int_{\mathbb{R}} x^{k} \mu_{B_n}(dx)\right]=\frac{1}{N} \sum_{\sigma_{1}, \ldots, \sigma_{k} \in K^{d-1}} \mathbb{E}\left[B_{\sigma_{1} \sigma_{2}} B_{\sigma_{2} \sigma_{3}} \ldots B_{\sigma_{k-1} \sigma_{k}} B_{\sigma_{k} \sigma_{1}}\right]
\end{equation}
where $B_{\sigma_i \sigma_{i+1}}$ is an element of the matrix $B_n$ indexed by $(d-1)$ cells and $N=\operatorname{dim}(B_n)$. By the definition of  $B_n$, $B_{\sigma_i \sigma_{i+1}}$ is non-zero only if $\sigma_i \cup \sigma_{i+1}$ is a $d$-cell in $Y$. As a result, the summation in (\ref{eqn:k-th moment exp1}) can be restricted to the summation over closed words of length $k+1$. Furthermore, for distinct $d$-cells $\tau,\tau^\prime$ and integers $i,j$ such that $\sigma_i \cup \sigma_{i+1}=\tau,\ \sigma_j \cup \sigma_{j+1}=\tau^\prime, $ the random variables  $B_{\sigma_i \sigma_{i+1}}$ and $B_{\sigma_j \sigma_{j+1}}$ are independent. Thus, we get
\begin{equation}\label{eqn:k-th moment exp2}
\mathbb{E}\left[\int_{\mathbb{R}} x^{k} \mu_{B_n}(dx)\right]=\frac{1}{N}\sum_{\substack{\text{closed words } w \\ \text{of length } k+1}} \prod_{\tau \in K^{d}} \mathbb{E}\left[(\chi-p)^{N_{w}(\tau)}\right],
\end{equation}
where $\chi$ is a Bernoulli random variable with parameter $p$ and $N_w(\tau)$ is as defined in Definition \ref{defn: word}.

Recall the equivalence relation defined on the set of words in Definition \ref{defn:equi_word}. For a word $w$ in (\ref{eqn:k-th moment exp2}), let $[w]$ denote the equivalence class containing $w$. For a closed word $w$, define
\begin{equation*}
T(w):= \prod_{\tau \in K^{d}} \mathbb{E}\left[(\chi-p)^{N_{w}(\tau)}\right].
\end{equation*}
Notice that $T(w^{\prime})= T(w)$ for all $w^{\prime} \in [w]$. 

Furthermore, if $N_w(\tau)=1$ for some $\tau \in K^d$, then $\mathbb{E}\left[(\chi-p)^{N_{w}(\tau)}\right]=0$. Also, $\mathbb{E}\left[(\chi-p)^{N_{w}(\tau)}\right]=1$ if $N_w(\tau)=0$ for some $\tau \in K^d$. Thus, the summation in (\ref{eqn:k-th moment exp2}) can be restricted to summation over closed words $w$ of length $k+1$ such that $N_w(\tau) \neq 1$ for all $\tau \in K^d$. 

Consider a word $w$ such that $N_w(\tau) \neq 1$ for all $\tau \in K^d$. Then $|\operatorname{supp}_d(w)|$ is bounded above by $\lfloor k/2 \rfloor$ and thus by Lemma \ref{claim: supp_0,supp_d}, $|\operatorname{supp}_0(w)|$ is bounded above by $\lfloor k/2 \rfloor + d$. Therefore (\ref{eqn:k-th moment exp2})  can be written as 
\begin{align}\label{eqn:k-th moment exp3}
\mathbb{E}\left[\int_{\mathbb{R}} x^{k} \mu_{B_n}(dx)\right]&=\frac{1}{N}\sum_{[w]:\substack{ N_w(\tau)\neq 1 \\ \forall \tau \in K^d}}|[w]| T(w)\nonumber\\
&=\frac{1}{N}\sum_{s=d+1}^{\left\lfloor\frac{k}{2}\right\rfloor+d} \sum_{[w] \in \mathcal{W}_{s}^{k}}|[w]| T(w) \nonumber\\
&= \frac{1}{N}\sum_{s=d+1}^{\left\lfloor\frac{k}{2}\right\rfloor+d} \left(\sum_{[w] \in \widetilde{\mathcal{W}}_{s}^{k}}|[w]| T(w)+\sum_{[w] \in \mathcal{W}_s^k\setminus \widetilde{\mathcal{W}}_s^k}|[w]| T(w)\right).
\end{align}

 First, note that $N={n \choose d}= O(n^d)$. Next consider the term $T(w)$ for $w \in \mathcal{W}_s^k$. Substituting the expression for moments of the Bernoulli random variable $\chi$, we get 
\begin{align}\label{eqn: prod_expectation}
	T(w)
	&= 
	\prod_{\tau \in K^{d} \atop N_{w}(\tau) \geq 2} p(1-p)\left[(1-p)^{N_{w}(\tau)-1}+(-1)^{N_{w}(\tau)} p^{N_{w}(\tau)-1}\right].
\end{align}
As $np \rightarrow \lambda>0$, $p^{N_w(\tau)-1} \rightarrow 0$ and $(1-p)^{N_{w}(\tau)-1} \rightarrow 1$. Thus each term in the product inside (\ref{eqn: prod_expectation}) is of the order $O(p)=O(1/n)$. Therefore for $[w] \in \mathcal{W}_s^k$,
\begin{align*}
	T(w)
	&= O\left(\frac{1}{n^{|\{\tau:N_{w}(\tau) \geq 2\}|}}\right)\\
	&= O\left( \frac{1}{n^{|\operatorname{supp}_d(w)|}}\right)= O\left( p^{|\operatorname{supp}_d(w)|}\right).
\end{align*}

Consider $[w] \in \mathcal{W}_s^k \setminus \widetilde{\mathcal{W}}_s^k$ for some $d+1 \leq s \leq \lfloor k/2 \rfloor +d$. Note that the number of different choices of $0$-cells in $\operatorname{supp}_0(w)$ is bounded by $n^s$ and therefore $|[w]| \leq n^s$.
Lemma \ref{claim:size W_k,s} implies that the summation in (\ref{eqn:k-th moment exp3}) is a finite sum. Furthermore, note that since $[w] \notin \widetilde{\mathcal{W}}_s^k$, we get from Lemma \ref{claim: supp_0,supp_d} that $|\operatorname{supp}_d(w)|>s-d$. As a result, it follows that
\begin{equation*}
	\frac{1}{N}\sum_{s=d+1}^{\left\lfloor\frac{k}{2}\right\rfloor+d} \sum_{[w] \in \mathcal{W}_s^k\setminus \widetilde{\mathcal{W}}_s^k}\hspace{-3mm}|[w]| T(w)=\hspace{-1mm} \sum_{s=d+1}^{\left\lfloor\frac{k}{2}\right\rfloor+d}O\left(\frac{1}{n^d} \times n^s \times \frac{1}{n^{|\operatorname{supp}_d(w)}|}\right)\hspace{-1mm}=\hspace{-1mm} \sum_{s=d+1}^{\left\lfloor\frac{k}{2}\right\rfloor+d}O(n^{s-d-|\operatorname{supp}_d(w)|})=o(1).
\end{equation*}

For $[w] \in \widetilde{\mathcal{W}}_s^k$, $|\operatorname{supp}_d(w)|=s-d$ and from Lemma \ref{claim:card_[w]} we have that $|[w]|=\frac{n!}{(n-s)!d!}$. This implies that
\begin{align}
\lim_{n \rightarrow \infty}\mathbb{E}\left[\int_{\mathbb{R}} x^{k} \mu_{B_n}(dx)\right]&= \lim_{n \rightarrow \infty} \sum_{s=d+1}^{\left\lfloor\frac{k}{2}\right\rfloor+d} \frac{\frac{n!}{(n-s)!d!}}{\frac{n!}{d!(n-d)!}}\sum_{w \in \widetilde{\mathcal{W}}_s^k} T(w)\nonumber\\
&= \lim_{n \rightarrow \infty}\sum_{s=d+1}^{\left\lfloor\frac{k}{2}\right\rfloor+d}\frac{n^s}{n^d}\sum_{w \in \widetilde{\mathcal{W}}_s^k}p^{s-d}\nonumber\\
&=\sum_{s=d+1}^{\left\lfloor\frac{k}{2}\right\rfloor+d} \lambda^{s-d}\left|\widetilde{\mathcal{W}}_s^k\right| = \beta_k \label{eqn: moment_final}.
\end{align}
Since $\beta_{k}$ is the limiting moment of the sequence of $k$-th moments of $\mu_{B_n}$, it follows that the sequence $\{\beta_k\}$ is a positive semi-definite sequence and thus there exist distributions on $\R$ with moment sequence $\{\beta_{k}\}$. By Lemma \ref{lemma: unique:beta_h} and the method of moments, we obtain that the expected empirical spectral distribution of $B_n$ converges to $\Gamma_d(\lambda)$. 
\end{proof}

Now we proceed to show the convergence of ESD using a concentration inequality based argument. In particular, we show that the $k$-th moment of $B_n$, $m_k(B_n)$ converge to $\beta_k$ almost surely.

\begin{lemma}\label{lemma: sum Var B_n}
For $d \geq 2, n \geq d+1$, let $B_n$ denote the unsigned adjacency matrix of $Y_d(n,p)$ and let $m_k(B_n)$ denote the $k-$th moment of $B_n$. Then for $np \rightarrow \lambda>0$,
$\lim_{n \rightarrow \infty} \mathbb{E} \left[m_k(B_n)- \mathbb{E}m_k(B_n)\right]^2 =O(1/n^d)$ for every positive integer $k$.
\end{lemma}
\begin{proof}
We have
\begin{align}\label{eqn: m_k(B_n)}
m_k(B_n)&=\frac{1}{N} \sum_{\sigma_{1}, \ldots, \sigma_{k} \in K^{d-1}} B_{\sigma_{1} \sigma_{2}} B_{\sigma_{2} \sigma_{3}} \ldots B_{\sigma_{k-1} \sigma_{k}} B_{\sigma_{k} \sigma_{1}} \nonumber\\
&=\frac{1}{N} \sum_{w}  \prod_{\tau \in K^d} B_{\tau}^{N_w(\tau)}, 
\end{align}
where the summation is over closed words of length $k+1$ on the vertex set $V=\{1,2,\ldots,n\}$, and $\{B_{\tau}\}_{\tau \in K^d}$ are independent random variables with distribution Ber($p$). From (\ref{eqn: m_k(B_n)}), we get
\begin{equation}\label{eqn: m_k(B_n) var}
\mathbb{E} \left[m_k(B_n)- \mathbb{E}m_k(B_n)\right]^2= \frac{1}{N^2} \sum_{w_1, w_2}  \left(\prod_{\tau \in K^d} \mathbb{E} B_{\tau}^{N_{w_1}(\tau)+N_{w_2}(\tau)}- \prod_{\tau \in K^d} \mathbb{E} B_{\tau}^{N_{w_1}(\tau)}\mathbb{E} B_{\tau}^{N_{w_2}(\tau)}\right),
\end{equation}
where $w_1, w_2$ are  closed words of length $k+1$ on the vertex set $V$. 

For closed words $w_1,w_2$, we introduce the notations, $\operatorname{supp}_0(w_1,w_2)= \operatorname{supp}_0(w_1) \cup \operatorname{supp}_0(w_2)$ and $\operatorname{supp}_d(w_1,w_2)= \operatorname{supp}_d(w_1) \cup \operatorname{supp}_d(w_2)$. We claim that for closed words $w_1, w_2$ such that $\operatorname{supp}_d(w_1) \cap \operatorname{supp}_d(w_2)$ is empty, the summand in (\ref{eqn: m_k(B_n) var}) is zero. To see this, notice that for any $\tau \in \operatorname{supp}_d(w_1,w_2)$ such that $\tau$ belong to only one of $\operatorname{supp}_d(w_1)$ or $\operatorname{supp}_d(w_2)$, $\mathbb{E} B_{\tau}^{N_{w_1}(\tau)+N_{w_2}(\tau)}= \mathbb{E} B_{\tau}^{N_{w_1}(\tau)}\mathbb{E} B_{\tau}^{N_{w_2}(\tau)}$. Thus, if $\operatorname{supp}_d(w_1) \cap \operatorname{supp}_d(w_2) = \emptyset$, 
$$\prod_{\tau \in K^d} \mathbb{E} B_{\tau}^{N_{w_1}(\tau)+N_{w_2}(\tau)}= \prod_{\tau \in K^d} \mathbb{E} B_{\tau}^{N_{w_1}(\tau)}\mathbb{E} B_{\tau}^{N_{w_2}(\tau)}.$$ Therefore the summand in (\ref{eqn: m_k(B_n) var}) is equal to zero in this case. Hence $(\ref{eqn: m_k(B_n) var})$ can be written as,

\begin{equation}\label{eqn: Var m_k, two terms}
\mathbb{E} \left[m_k(B_n)- \mathbb{E}m_k(B_n)\right]^2= \frac{1}{N^2} \hspace{-7mm} \sum_{w_1, w_2 \atop \operatorname{supp}_d(w_1) \cap \operatorname{supp}_d(w_2) \neq \phi } \hspace{-3mm} \left(\prod_{\tau \in K^d}\hspace{-2mm} \mathbb{E} B_{\tau}^{N_{w_1}(\tau)+N_{w_2}(\tau)}\hspace{-2mm} - \prod_{\tau \in K^d} \hspace{-2mm} \mathbb{E} B_{\tau}^{N_{w_1}(\tau)}\mathbb{E} B_{\tau}^{N_{w_2}(\tau)}\right).
\end{equation}

Consider $\tau \in \operatorname{supp}_d(w_1) \cap \operatorname{supp}_d(w_2)$. Since $B_\tau=(\chi_\tau-p)$, where $\chi_\tau$ is a Bernoulli random variables, it follows that $\mathbb{E}B_{\tau}^r=O(p)=O(1/n)$ for all $r \geq 1$. Hence $\mathbb{E}B_{\tau}^{N_{w_1}(\tau)+N_{w_2}(\tau)}=O(1/n)$ and $\mathbb{E}B_{\tau}^{N_{w_1}(\tau)}\mathbb{E}B_{\tau}^{N_{w_2}(\tau)}=O(1/n^2)$.

Consider $w_1,w_2$ such that $\operatorname{supp}_d(w_1) \cap \operatorname{supp}_d(w_2)$ is non-empty. It follows that for $\tau \in \operatorname{supp}_d(w_1) \cap \operatorname{supp}_d(w_2)$, $\mathbb{E}B_{\tau}^{N_{w_1}(\tau)}\mathbb{E}B_{\tau}^{N_{w_2}(\tau)}= O\left(\frac{1}{n} \times\mathbb{E}B_{\tau}^{N_{w_1}(\tau)+N_{w_2}(\tau)}\right)$. Hence, if 
\begin{equation}\label{eqn:B_n var iff}
	\lim_{n \rightarrow \infty} \frac{1}{N^2} \sum_{w_1, w_2 \atop \operatorname{supp}_d(w_1) \cap \operatorname{supp}_d(w_2) \neq \phi }  \prod_{\tau \in K^d} \mathbb{E} B_{\tau}^{N_{w_1}(\tau)+N_{w_2}(\tau)} = O\left(\frac{1}{n^d}\right),
\end{equation}
then $\lim_{n \rightarrow \infty}\mathbb{E} \left[m_k(B_n)- \mathbb{E}m_k(B_n)\right]^2 =O(1/n^d).$ Therefore, to prove the lemma, it is sufficient to prove (\ref{eqn:B_n var iff}). We make the following claim:\\\\
\textbf{Claim:}
 For closed words $w_1, w_2$ of length $k+1$ with $\operatorname{supp}_d(w_1) \cap \operatorname{supp}_d(w_2) \neq \phi$
	\begin{equation*}
		d \leq |\operatorname{supp}_0(w_1,w_2)| \leq |\operatorname{supp}_d(w_1,w_2)|+d.
	\end{equation*} 
\textit{Proof of Claim:} It is clear that $d \leq |\operatorname{supp}_0(w_1,w_2)|$ and therefore it is only required to prove the second inequality. Consider two words $w_1=\sigma_{1}^{(1)} \sigma_{2}^{(1)} \cdots \sigma_{k+1}^{(1)}$ and $w_2=\sigma_{1}^{(2)} \sigma_{2}^{(2)} \cdots \sigma_{k+1}^{(2)}$ such that $w_1,w_2$ are closed and $\operatorname{supp}_d(w_1) \cap \operatorname{supp}_d(w_2) \neq \phi$.
We start by counting the number of elements in $\operatorname{supp}_0(w_1)$ and $\operatorname{supp}_d(w_1)$. From Lemma \ref{claim: supp_0,supp_d}, we have 
 \begin{equation}\label{eqn:supp_d w1}
  |\operatorname{supp}_0(w_1)| \leq |\operatorname{supp}_d(w_1)|+d.
 \end{equation}

Now we count the additional elements in $\operatorname{supp}_0(w_2)$ and $\operatorname{supp}_d(w_2)$ that have not been counted yet. Consider a $d$-cell $\tau \in \operatorname{supp}_d(w_1) \cap \operatorname{supp}_d(w_2)$ and let $i$ be the smallest integer such that $\sigma_i^{(2)}\cup\sigma_{i+1}^{(2)}=\tau$. Note that as $\tau \in \operatorname{supp}_d(w_1)$, the $d$-cell $\tau$ is already counted once. Furthermore, this also implies that all the 0-cells in $\tau$ has also been counted.

To count the remaining $0$-cells and $d$-cells, we start by counting the new $0$-cells in $\sigma_{i-1}^{(2)}$ and proceed in descending order till $\sigma_{1}^{(2)}$. Note that if $\tau^\prime=\sigma_{j}^{(2)} \cup \sigma_{j-1}^{(2)}$ contains a new $0$-cell that is not yet counted, then the $d$-cell $\tau^\prime$ has also not been counted yet. After reaching $\sigma_{1}^{(2)}$, we start counting the new $0$-cells in $\sigma_{i+2}^{(2)}$ and proceed in ascending order till $\sigma_{k+1}^{(2)}$. Note that in this case, if  $\tau^\prime=\sigma_{j}^{(2)} \cup \sigma_{j+1}^{(2)}$ contains a new $0$-cell, then the $d$-cell $\tau^\prime$ is also new. In short, for all new $0$-cell that appears in $\operatorname{supp}_0(w_2)$, associated to it, a new $d$-cell appears in $\operatorname{supp}_d(w_2)$. Using this information along with (\ref{eqn:supp_d w1}), we get
$|\operatorname{supp}_0(w_1,w_2)| \leq |\operatorname{supp}_d(w_1,w_2)|+d$.
 \hfill\qedsymbol
\\\\
We now define the equivalence between tuples of words $(w_1,w_2)$ and $(w_1^\prime,w_2^\prime)$. Two 2-tuples of words $(w_1,w_2)$ and $(w_1^\prime,w_2^\prime)$ are said to be equivalent if there exists a map $\pi: \operatorname{supp}_0(w_1,w_2) \rightarrow  \operatorname{supp}_0(w_1^\prime
,w_2^\prime) $ such that the restrictions $ \pi\big|_{\operatorname{supp}_0(w_1)}$ defines an equivalence between $w_1$ and $w_1^\prime$, and  $\pi\big|_{\operatorname{supp}_0(w_2)}$ is a bijection from $\operatorname{supp}_0(w_2)$ to $\operatorname{supp}_0(w_2^\prime)$. Now, note that the left hand side of (\ref{eqn:B_n var iff}) can be written as 
\begin{align*}
\frac{1}{N^2} \hspace{-4mm}\sum_{w_1, w_2 \atop \operatorname{supp}_d(w_1) \cap \operatorname{supp}_d(w_2) \neq \phi }  \prod_{\tau \in K^d} \mathbb{E} B_{\tau}^{N_{w_1}(\tau)+N_{w_2}(\tau)} &= \sum_{s=d}^{k+d} \sum_{(w_1,w_2) \in W_2(k,s,d)} \hspace{-4mm}\frac{|[(w_1,w_2)]|}{N^2}\prod_{\tau \in K^d} \mathbb{E} B_{\tau}^{N_{w_1}(\tau)+N_{w_2}(\tau)},
\end{align*}
 where $W_2(k,s,d)$ is the set of all equivalence classes of words $(w_1,w_2)$ such that $w_1,w_2$ are closed words of length $k+1$, $|\operatorname{supp}_0(w_1,w_2)|=s $ and $\operatorname{supp}_d(w_1) \cap \operatorname{supp}_d(w_2) \neq \phi$.
 
Since $\mathbb{E} B_{\tau}^{N_{w_1}(\tau)+N_{w_2}(\tau)}= O(1/n)$ for all $\tau$ and all choices of $N_{w_1}(\tau)+N_{w_2}(\tau)$, it follows that $\prod_{\tau \in K^d} \mathbb{E} B_{\tau}^{N_{w_1}(\tau)+N_{w_2}(\tau)}= O(n^{-|\operatorname{supp}_d(w_1,w_2)|})$. Furthermore, for each $(w_1,w_2) \in W_2(k,s,d)$, $|[(w_1,w_2)]| \leq O(n^s)$. Thus for each fixed $s$, the summand is of order $O\left(\frac{n^s}{n^{2d}} \times n^{-|\supp_d(w_1,w_2)|}\right)$. 

By the claim, we have $s-|\operatorname{supp}_d(w_1,w_2)|-2d \leq -d$ and therefore
 \begin{equation*}
 	 \sum_{s=d}^{k+d} \sum_{(w_1,w_2) \in W_2(k,s,d)} \frac{[(w_1,w_2)]}{N^2}\prod_{\tau \in K^d} \mathbb{E} B_{\tau}^{N_{w_1}(\tau)+N_{w_2}(\tau)}= O(n^{-d}).
 \end{equation*}
This completes the proof of the lemma.
\end{proof}
\begin{proof}[Proof of Theorem \ref{thm: centred LSD} (i)]
	From Theorem \ref{thm:EESD_centred_unsigned}, it follows that for all $k$, the $k$-th moment of the EESD, $m_k(\mathbb{E} F_{B_n})$ converges to $\beta_{k}(\lambda)$ as $np \rightarrow \lambda$. By Lemma \ref{thm:EESD_centred_unsigned} and Borel-Cantelli lemma, it follows that the $k$-th moment of ESD, $m_k( F_{B_n})$ converge almost surely to $\beta_{k}(\lambda),$ implying the required result.
\end{proof}

\subsubsection{Proof of Theorem \ref{thm: centred LSD}: Signed adjacency matrix}\label{subsubsec: Proof of centred signed}

 We introduce a definition and a lemma required to prove the convergence of EESD of the signed adjacency matrix. The proof of this lemma is pushed to the next subsection.
\begin{definition}\label{defn:sgn}
	Let $Y$ be a random simplicial complex of dimension $d$, $w$ be a closed word of length $k+1$ and $\tau \in \operatorname{supp}_d(w)$, we define the set $\widehat{E}_w(\tau)=\{(i,i+1):\sigma_i \cup\sigma_{i+1}=\tau \text{ and } \sigma_i \stackrel{Y}{\sim} \sigma_{i+1}\}$. Further, we define
	\begin{align*}
		&\sgn(w,\tau)= (-1)^{ |\widehat{E}_w(\tau)|} \text{ and }\\
		&\sgn(w)= \prod_{\tau \in \operatorname{supp}_d(w)}(-1)^{ |\widehat{E}_w(\tau)|}.
	\end{align*} 
\end{definition}
\begin{lemma}\label{lemma: sgn(w)=(-1)^k}
	Let $d \geq 2, k \in \mathbb{Z}^+$ and $d+1 \leq s \leq \lfloor \frac{k}{2}\rfloor +d$, consider a closed word $w$ of length $k+1$ such that $|\operatorname{supp}_0(w)|=s$ and $|\operatorname{supp}_d(w)|=s-d$. Then 
	$\sgn(w)=(-1)^k$.
\end{lemma}
Now we prove the convergence of expected empirical spectral distribution for centred signed adjacency matrices.

\begin{theorem}\label{thm:signed EESD convergence}
	For $np \rightarrow \lambda>0$, $\mathbb{E}F_{B_n^+}$, the expected empirical spectral distribution of the centred signed adjacency matrix $B_n^+$ of $Y_d(n,p)$, converges  weakly to $-\Gamma_d(\lambda)$.
\end{theorem}
\begin{proof}
	Note that for even $k$, the $k$-th moment of $-\Gamma_d(\lambda)$ is $\beta_k$ and for odd $k$, the $k-$th moment is $-\beta_k$. Hence by the method of moments, it is sufficient to show that 
	\begin{equation*}
		\lim_{n \rightarrow \infty}\mathbb{E}\left[\int_{\mathbb{R}} x^{k} \mu_{B_n^+}(dx)\right]= (-1)^k \beta_k \text{ for all } k.
	\end{equation*} 
	By the same reasoning employed in Theorem \ref{thm:EESD_centred_unsigned} to obtain (\ref{eqn:k-th moment exp2}), we get,
	\begin{align*}
		\mathbb{E}\left[\int_{\mathbb{R}} x^{k} \mu_{B_n^+}(dx)\right] &=\frac{1}{N}\sum_{\substack{\text{closed words } w :\\ N_w(\tau)\neq 1, \forall \tau \in K^d}} \prod_{\tau \in K^{d}} \mathbb{E}\left[(\chi-p)^{N_{w}(\tau)}\right]\sgn(w,\tau) \\
		&= \frac{1}{N}\sum_{\substack{\text{closed words } w :\\ N_w(\tau)\neq 1, \forall \tau \in K^d}} \left( \prod_{\tau \in K^{d}} \mathbb{E}\left[(\chi-p)^{N_{w}(\tau)}\right] \prod_{\tau \in K^{d}}\sgn(w,\tau) \right),
	\end{align*}
	where $\chi$ is a Bernoulli random variable with parameter $p$ and $\sgn(w,\tau)$ is as defined in Definition \ref{defn:sgn}. 
	Further calculations imply that the contribution due to all closed words $w$ such that $|\operatorname{supp}_d(w)| \neq |\operatorname{supp}_o(w)|-d$, is of the order $o(1)$. Now, observe that $\prod_{\tau \in K^{d}}\sgn(w,\tau)= \sgn(w)$ and therefore by Lemma \ref{lemma: sgn(w)=(-1)^k}, we get that 
	\begin{equation}\label{eqn:k-th moment_signed}
		\mathbb{E}\left[\int_{\mathbb{R}} x^{k} \mu_{B_n^+}(dx)\right]= \frac{(-1)^k}{N}\sum_{\substack{\text{closed words } w :\\ N_w(\tau)\neq 1, \forall \tau \in K^d \atop |\operatorname{supp}_d(w)| \neq |\operatorname{supp}_o(w)|-d}}\prod_{\tau \in K^{d}} \mathbb{E}\left[(\chi-p)^{N_{w}(\tau)}\right].
	\end{equation}
	From the calculations in (\ref{eqn:k-th moment exp2}), it follows from here that the right-hand side of (\ref{eqn:k-th moment_signed}) converge to $(-1)^k \beta_k$ and hence the proof is complete.
\end{proof}
\begin{lemma}\label{lemma: sum Var B_n signed}
	For $d \geq 2, n \geq d+1$ and $k \in \N$, let $B_n^+$ denote the centred signed adjacency matrix of $Y_d(n,p)$ and let $m_k(B_n^+)$ denote the $k-$th moment of $B_n^+$. Then for $np \rightarrow \lambda>0$,
	$\lim_{n \rightarrow \infty} \mathbb{E} \left[m_k(B_n^+)- \mathbb{E}m_k(B_n^+)\right]^2 =O(1/n^d)$ for every positive integer $k$.
\end{lemma}
\begin{proof}[Proof of Theorem \ref{thm: centred LSD} (ii)]
	By a similar argument used in the proof of Theorem \ref{thm: centred LSD} $(i)$, Theorem \ref{thm:signed EESD convergence} and Lemma \ref{lemma: sum Var B_n signed} the proof follows.
\end{proof}
\noindent Now the only part left in proving Theorem \ref{thm: centred LSD} is the proof of Lemma \ref{lemma: sgn(w)=(-1)^k}.\\\\
\textit{Proof of Lemma \ref{lemma: sgn(w)=(-1)^k}}:
We first make a few observations about the signs of entries of the signed adjacency matrix $A_n^+$. Recall $K_+^{d-1}$, the set of all positively oriented $(d-1)$-cells on $n$ vertices. Let $\sigma, \sigma^\prime \in K_+^{d-1}$ be such that $\tau=\sigma \cup \sigma^\prime$ is a $d$-cell where the vertices in $\tau$ are given by $x^{(0)}< x^{(1)}<\cdots<x^{(d)} $. Note that both $\sigma$ and $\sigma^\prime$ can be obtained by removing $0$-cells from $\tau$. Suppose $\tau \setminus \sigma = \{x^{(i)}\}$ and $\tau \setminus \sigma^\prime = \{x^{(i^\prime)}\}$. Recall that the orientation induced on $\sigma$ by the positive orientation on $\tau$ is $(-1)^i[x^{(0)}, x^{(1)}, \ldots, x^{(i-1)}, x^{(i+1)}, \ldots ,x^{(d)}]$ and the orientation induced on $\sigma^\prime$ is $(-1)^{i^\prime}[x^{(0)}, x^{(1)}, \ldots ,x^{(i^\prime-1)} ,x^{(i^\prime+1)} ,\ldots, x^{(d)}]$. Since both $[x^{(0)}, x^{(1)}, \ldots, x^{(i-1)}, x^{(i+1)}, \ldots ,x^{(d)}]$ and $[x^{(0)}, x^{(1)}, \ldots ,x^{(i^\prime-1)}, x^{(i^\prime+1)}, \ldots , x^{(d)}]$ are positively oriented, it follows that the orientation induced on $\sigma$ (similarly $\sigma^\prime$) by $\tau$ is positive if and only if  $i$ (or similarly $i^\prime$) is even. Note that in Definition \ref{defn:adjacency matrix}, the matrix entries are indexed with $(d-1)$-cells with positive orientation and $\sigma \stackrel{Y}{\sim} \sigma^\prime$ if and only if $\sigma$ and $\bar{\sigma^\prime}$ are in the boundary of $\tau$ as oriented $(d-1)$-cells. Since $\sigma,\sigma^\prime$ are positively oriented, it follows that $\sigma \stackrel{Y}{\sim} \sigma^\prime$ if and only if both $i$ and $i^\prime$ have different parity. Thus we get,
\begin{equation}\label{eqn: sigma parity}
	(A_n^+)_{\sigma\sigma^\prime}=
	\begin{cases}
		1 &\text{ if } i \text{ and } i^\prime \text{ have different parity and } \sigma \cup \sigma^\prime \in Y,\\
		-1 &\text{ if } i \text{ and } i^\prime \text{ have same parity and } \sigma \cup \sigma^\prime \in Y,\\
		0 & \text{ otherwise.}
		\end{cases}
\end{equation}

 Consider a closed word $w= \sigma_1 \sigma_2 \cdots \sigma_{k+1}$ such that $|\operatorname{supp}_0(w)|=s$ and $|\operatorname{supp}_d(w)|=s-d$ for some $s$. For a $d$-cell $\tau \in \operatorname{supp}_d(w)$ and $m \geq 1$, we define the $m$-th entry time  for $\tau$ and the $m$-th exit time for $\tau $ as
\begin{align*}
	S_{m}(\tau) &=\inf \{i > T_{m-1}(\tau): \sigma_i \cup \sigma_{i+1}= \tau\},\\
T_m(\tau)&= \inf \{i > S_{m}(\tau): \sigma_{i} \cup \sigma_{i+1} \neq \tau\}.
\end{align*}
For $\tau= \sigma_{k} \cup \sigma_{k+1}$, we define the last exit point of $\tau$ as $T(\tau)=k+1$ and for all other $d$-cells $\tau$, the last exit point of $\tau$ is defined as $T(\tau)=\max\{T_m(\tau): T_m(\tau)<\infty\}$. We follow
the conventions $T_0(\tau)=0$ and $\inf \phi=\infty$.\\

\begin{figure}[h!]
	\hspace{-20mm}	\includegraphics[height=65mm, width =170mm]{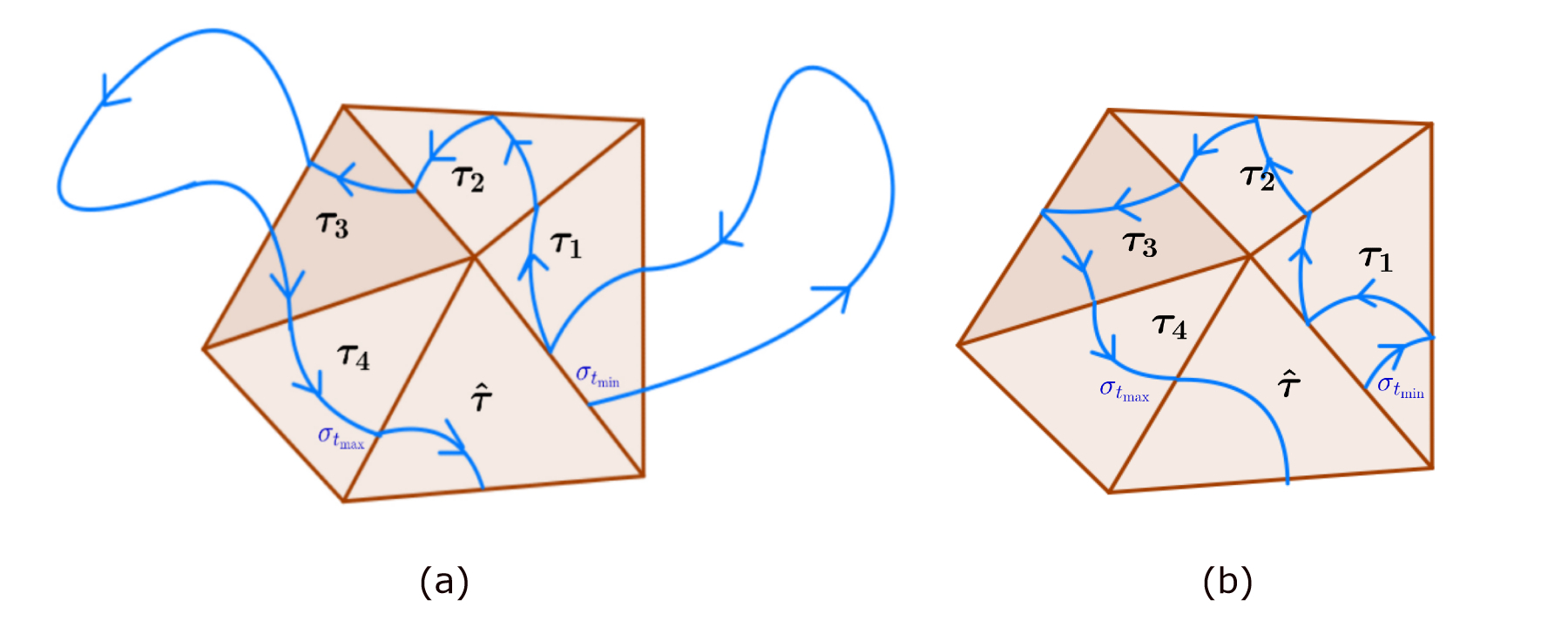}
	\caption{(a) shows a word $w^\prime$ in the simplicial complex such that $\sigma_{t_{\min}}\neq \sigma_{t_{\max}}$. (b) shows the word $w^{\prime \prime}$ obtained from $w^\prime$ after removing the bushes.}
	\label{fig:example_sign}
\end{figure} 
\noindent We make the following two claims:\\\\
\textbf{Claim 1:} For all $\tau \in \operatorname{supp}_d(w)$ and $T_m(\tau), S_{m+1}(\tau)<\infty$, $\sigma_{T_m(\tau)}=\sigma_{S_{m+1}(\tau)}.$\\\\
\textit{Proof of Claim1}: We prove the claim using proof by contradiction.
Consider the set 
\begin{equation*}
	\nabla=\{S_{m+1}(\tau): \tau \in \operatorname{supp}_d(w), m \geq 1 \text{ and } \sigma_{T_m (\tau)} \neq \sigma_{S_{m+1}(\tau)}, T_m(\tau),S_{m+1}(\tau)<\infty\}.
\end{equation*}

Suppose the claim is false, then $\nabla$ is non-empty and there exists $M >0$ and $\hat{\tau} \in \operatorname{supp}_d(w)$ such that
\begin{equation*}
	\min \nabla=S_{M+1}(\hat{\tau}).
\end{equation*}
We define 
\begin{equation}\label{tmin, tmax}
	t_{\min}= T_M(\hat{\tau}) \text{ and } t_{\max}= S_{M+1}(\hat{\tau}).
\end{equation}
Recall the notion of subword defined in Definition \ref{defn:subword}. Consider the subword $w^\prime$ of $w$ defined by $w^\prime= \sigma_{t_{\min}}\sigma_{t_{\min}+1}\cdots \sigma_{t_{\max}}$. 

 
 We define a bush as a subword, $\hat{w}$ of $w^\prime$ such that $\hat{w}=\sigma_{i_1}\sigma_{i_2}\cdots\sigma_{i_m}$, where $i_1=T_r(\tau^\prime)$ and $i_m=S_{r+1}(\tau^\prime)$ for some $r,\tau^\prime$ such that $t_{\min} < T_r(\tau^\prime) < S_{r+1}(\tau^\prime)< t_{\max}$. From the minimality of $S_{M+1}(\tau)$, it follows that each bush is a closed word. We define $w^{\prime \prime}$ as the word obtained from $w^\prime$ after the deletion of all bushes in $w^\prime$, i.e., for $t_{\min} \leq i \leq t_{\max}$, $\sigma_i$ is deleted from $w^\prime$ if there exist $r,\tau^\prime$ such that $t_{\min} <T_r(\tau^\prime) < i \leq S_{r+1}(\tau^\prime) < t_{\max}$ (see Figure \ref{fig:example_sign} (b)). Since $\sigma_{T_r(\tau^\prime)}=\sigma_{S_{r+1}(\tau^\prime)}$, it follows that $w^{\prime \prime}$ is a word.

Let $\tau_1, \tau_2, \ldots , \tau_u=\hat{\tau}$ be the elements of $\operatorname{supp}_d(w^{\prime \prime})$ in the chronological order of their 
appearance in $w^{\prime \prime}$. Note that for each $ 1 \leq i \leq u$, $\tau_i \subset \tau_{i-1} \cup \tau_{i+1}$ where $\tau_0:= \tau_u$ and $\tau_{u+1}:=\tau_1$. Therefore, we get that for each $i$,
\begin{equation}\label{eqn: tau_i union}
	\tau_i \subset \bigcup_{j\neq i} \tau_j.
\end{equation}

Let $\tau_m$ be such that $S_1(\tau_m)= \max \{S_1(\tau_1),S_1(\tau_2),\ldots , S_1(\tau_u)\}$. Consider the subword $\tilde{w}_{S_1(\tau_m)-1}= \sigma_1 \sigma_2 \cdots \sigma_{S_1(\tau_m)-1}$. Then note that all $d$-cells except $\tau_m$ of $\{\tau_1,\tau_2,\ldots ,\tau_u\}$ belongs to $\operatorname{supp}_d(\tilde{w}_{S_1(\tau_m)-1})$. Therefore from (\ref{eqn: tau_i union}), it follows that  
\begin{equation}\label{eqn:tau union 2}
	 \bigcup_{j=1}^u \tau_j \subset \operatorname{supp}_0(\tilde{w}_{S_1(\tau_m)-1}).
\end{equation}
Now since $|\operatorname{supp}_0(w)|=s$ and $|\operatorname{supp}_d(w)|=s-d$ it follows from (\ref{eqn:iff W_s^k}) that, $|\operatorname{supp}_d(\tilde{w}_{S_1(\tau_m)-1})|=|\operatorname{supp}_0(\tilde{w}_{S_1(\tau_m)-1})|-d$. Note that the $d$-cell $\tau_m$ is appearing for the first time in $\tilde{w}_{S_1(\tau_m)}$ and therefore $|\operatorname{supp}_d(\tilde{w}_{S_1(\tau_m)})|=|\operatorname{supp}_d(\tilde{w}_{S_1(\tau_m)-1})|+1$. But applying (\ref{eqn:tau union 2}), we get $|\operatorname{supp}_0(\tilde{w}_{S_1(\tau_m)})|=|\operatorname{supp}_0(\tilde{w}_{S_1(\tau_m)-1})|$ and consequently $|\operatorname{supp}_d(\tilde{w}_{S_1(\tau_m)})|=|\operatorname{supp}_0(\tilde{w}_{S_1(\tau_m)-1})|-d+1$, a contradiction to (\ref{eqn:iff W_s^k}). Hence the claim is proved. \hfill\qedsymbol \\


\noindent \textbf{Claim 2:} For $\tau \in \operatorname{supp}_d(w)$, $\sigma_{S_{1}(\tau)}=\sigma_{T(\tau)}$.\\
\begin{figure}
	\includegraphics[height=80mm, width=130mm]{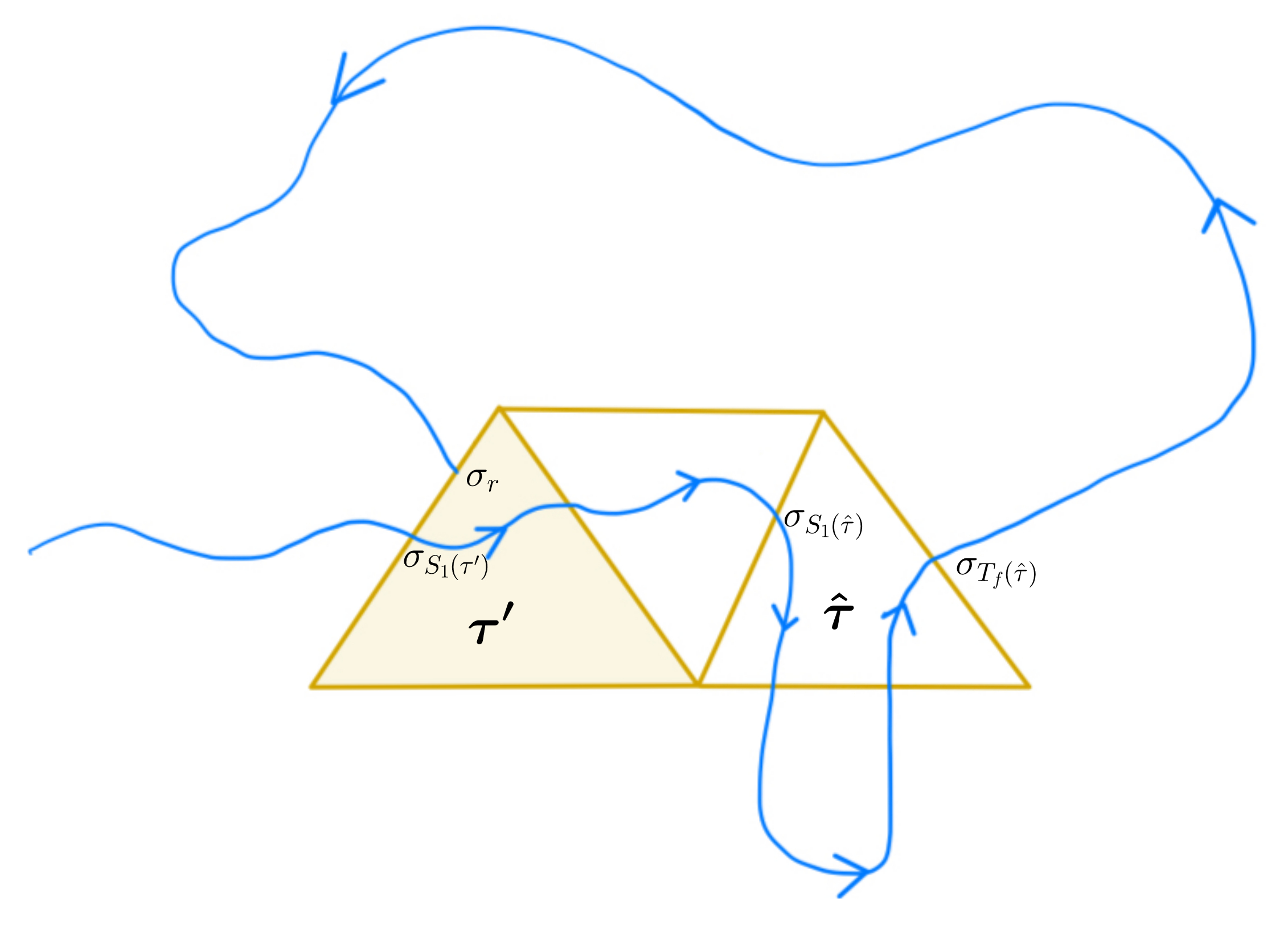}
	\caption{an illustration of a word $w$ with $d$-cell $\tau$ such that $\sigma_{S_1(\hat{\tau})} \neq \sigma_{T(\hat{\tau})}$ and $r=\min N$. }\label{fig:ex_final_exit}
\end{figure}

\noindent\textit{Proof of Claim 2}: Note that for $\tau= \sigma_1 \cup \sigma_2$, the case is trivial. For the rest of cases, the proof of Claim 2 is similar to the proof of Claim 1 and therefore, we outline only the major steps. Suppose the claim is false. Then the set
\begin{equation*}
	\nabla= \{T(\tau): T(\tau)\neq S_1(\tau), \tau \in \operatorname{supp}_d(w)\}
\end{equation*} is non-empty. Let $M,\hat{\tau}$ be such that $\max \nabla= T_M(\hat{\tau})$ and we define
for $1 \leq i \leq k+1$, consider the subword $\tilde{w}_i=\sigma_{1}\sigma_{2}\cdots \sigma_i$. Consider the set
\begin{equation*}
	N=\{i: T_{M}(\hat{\tau}) \leq i \leq k, \sigma_i \subseteq \tau^\prime \text{ for some } \tau^\prime \in \operatorname{supp}_d(\tilde{w}_{t_{S_1(\hat{\tau})}})\}
\end{equation*} 
Since $\sigma_1 = \sigma_{k+1}$, it follows that $N$
is non-empty. Let  $r$ be the smallest element in $N$ and let $\tau^\prime$ be and let $\sigma_r \subseteq \tau^\prime$ where $\tau^\prime \in \operatorname{supp}_d(\tilde{w}_{t_{S_1(\hat{\tau})}})$. Define $w^{\prime \prime}$ as the word obtained from $w^\prime= \sigma_{S_1(\hat{\tau})}\sigma_{S_1(\hat{\tau})+1}\cdots \sigma_{r}$ after the removal of bushes. Then the $d$-cells $\tau_1,\tau_2,\ldots,\tau_u$ belonging to $\operatorname{supp}_d(w^{\prime \prime})$ obeys (\ref{eqn: tau_i union}) and the rest of the argument proceeds similar to Claim 1. For a pictorial representation of this case, see Figure \ref{fig:ex_final_exit}. \qed

Now, using the above claims we calculate $\sgn(w)$. 

Consider $ \tau=\{x^{(0)},x^{(1)}, \ldots, x^{(d+1)}\} \in \operatorname{supp}_d(w)$ where $x^{(i)}<x^{(j)}$ if $i<j$. Let $\sigma_i \subset \tau $ be the $(d-1)$-cell such that $\tau \setminus \sigma=\{x^{(i)}\}$. We define two sets
\begin{align*}
	\mathcal{E}_\tau &= \{\sigma : \tau \setminus \sigma=\{x^{(i)}\} \text{ where } i \text{ is even} \},\\
	\mathcal{O}_\tau &= \{\sigma : \tau \setminus \sigma=\{x^{(i)}\} \text{ where } i \text{ is odd} \}.
\end{align*}

\noindent Furthermore, define the set
\begin{align*}
	\mathcal{R}_\tau=\{i: \sigma_i \cup \sigma_{i+1}=\tau\}.
\end{align*}
Clearly, $\mathcal{R}_\tau=\{S_1(\tau),S_1(\tau)+1, \ldots , T_1(\tau),S_2(\tau),S_2(\tau)+1,\ldots, T_2(\tau),\ldots ,T(\tau)\}$, where $T(\tau)$ is the last exit point of $\tau$.

From (\ref{eqn: sigma parity}), it follows that for $i \in \mathcal{R}_\tau$, $(A_n)_{\sigma_i \sigma_{i+1}}=-1$ if $\sigma_i, \sigma_{i+1} \in \mathcal{E}_\tau $ or $\sigma_i, \sigma_{i+1} \in \mathcal{O}_\tau $ and $(A_n)_{\sigma_i \sigma_{i+1}}=1$ if one of $\{\sigma_i, \sigma_{i+1}\}$ belongs to $ \mathcal{E}_\tau $ and the other belongs to $ \mathcal{O}_\tau $.

Suppose $\mathcal{R}_\tau=\{r_1,r_2,\ldots, r_\ell\}$ where $r_i<r_j$ if $i<j$. For $r_i \in \mathcal{R}_\tau$, $(r_i,r_{i+1})$ is called a same-parity jump if  $\sigma_{r_i}, \sigma_{r_{i+1}} \in \mathcal{E}_\tau $ or $\sigma_{r_i}, \sigma_{r_{i+1}} \in \mathcal{O}_\tau $ and $(r_i,r_{i+1})$ is called a different-parity jump if one of $\{\sigma_{r_i}, \sigma_{r_{i+1}}\}$ belongs to $ \mathcal{E}_\tau $ and the other belongs to $ \mathcal{O}_\tau $. Recall $\sgn(w,\tau)$ from Definition \ref{defn:sgn}. It follows that, 
\begin{align}\label{eqn:sign(w,tau) diff. parity}
	\sgn(w,\tau) &=(-1)^{\#\{\text{ number of same-parity jumps } (r_i,r_{i+1}) : r_{i+1}=r_i+1 \}} \nonumber \\
	&=(-1)^{\#N_w(\tau)-\#\{\text{ number of different-parity jumps } (r_i,r_{i+1})\}},
\end{align}
where the last equation follows from Claim 1, by the observation that if $r_{i+1} \neq r_i +1$, then $(r_i,r_{i+1})$ is a same-parity jump.

Without loss of generality, suppose that $\sigma_{S_1(\tau)} \in \mathcal{E}_\tau$. Then by Claim 2, we have that $\sigma_{T(\tau)} \in \mathcal{E}_\tau$. Therefore, the number of different-parity jumps is always even. Hence it follows from (\ref{eqn:sign(w,tau) diff. parity}) that,
\begin{equation*}
	\sgn(w,\tau)=(-1)^{\#N_w(\tau)} \text{ for all } \tau \in \operatorname{supp}_d(w). 
\end{equation*}
Consequently, 
\begin{align*}
	\sgn(w)= \prod_{\tau \in \operatorname{supp}_d(w)}	\sgn(w,\tau)= \prod_{\tau \in \operatorname{supp}_d(w)}(-1)^{\#N_w(\tau)}=(-1)^k,
\end{align*}
where the last equality follows from the fact that $\sum N_w(\tau)$ is the total number of edges in $w$. Hence the lemma is proved.
\qed
\subsection{Proof of Theorem \ref{thm: existence_LSD}}\label{subsec:LSD for adjacency matrix}
In this section, we show that the empirical spectral distributions of the centred and non-centred adjacency matrices of $Y_d(n,p)$ are sufficiently close to converge to the same limit. We use the following well-known rank inequality to establish the closeness.
\begin{theorem}[Theorem A.43,\cite{bai_silverstein}]\label{thm:rank_inequality}
	Let $A$ and $B$ be two $n \times n$ Hermitian matrices, and let $F_A$ and $F_B$ be the empirical spectral distributions of $A$ and $B$, respectively. Then
	$$||F_A-F_B||_\infty \leq \frac{1}{n}\operatorname{Rank}(A-B).$$
\end{theorem}
In this section, we prove Theorem \ref{thm: existence_LSD} for the signed and unsigned adjacency matrices separately. First, we prove the theorem for the unsigned case.
\subsubsection{Proof of Theorem \ref{thm: existence_LSD}: Unsigned adjacency matrix}

Recall the expectation of unsigned adjacency matrix, given by $\mathbb{E} A$. 
It is easy to see that $\mathbb{E}A_n= p \mathbb{A}_n$, where $\mathbb{A}_n$ is the unsigned adjacency matrix of the $d$-dimensional complete simplicial complex  on $n$ vertices.
The proof for the unsigned adjacency matrix is similar to the proof of the signed case. Here we only state the relevant lemma. The following lemma follows easily from [Theorem 1,\cite{Johnson_graph}] as a special case.
\begin{lemma}\label{lemma:eig_unsigned complete}
	For $d \geq 2$ and $n \geq 2d$, the unsigned adjacency matrix $\mathbb{A}_n$ of the complete simplicial complex of dimension $d$ has $d+1$ distinct eigenvalues given by  
	$$\alpha_s=\sum_{r = s-1}^{\min\{s,d-1\}}(-1)^{s-r}{s \choose r}(d-r){n-d-s+r \choose 1-s+r}$$
	for $s=0,1,2,\ldots, d$ and the multiplicity of the eigenvalue $\alpha_s$ is ${n \choose s}-{n \choose s-1}$.
\end{lemma}
Now, we prove Theorem \ref{thm: existence_LSD} for the unsigned case.
\begin{proof}[Proof of Theorem \ref{thm: existence_LSD} (i)]
Let $F_{A_n+pdI}$ and $F_{B_n}$ denote the empirical spectral distribution of the matrices $A_n+pdI$ and $B_n$, respectively. First,  note that the eigenvalues of $A_n +pdI$ are the eigenvalues of $A_n$ translated by $pd$. For $np \rightarrow \lambda$, $pd \rightarrow 0$ and this implies that both $A_n$ and $A_n + pdI$ have the same limiting spectral distribution. Also,
$${A}_n + pdI = B_n + \mathbb{E}A_n + pdI= B_n + p\left(\mathbb{A}_n + dI\right).$$
From Lemma \ref{lemma:eig_unsigned complete}, it follows that $\mathbb{A}_n$ has eigenvalue $-d$ with multiplicity of the order $O(n^d)$ and the multiplicities of rest of the eigenvalues add up only to the order $O(n^{d-1})$, and therefore $\operatorname{Rank}(\mathbb{A}_n + dI)=O(n^{d-1})$. Now, from Theorem \ref{thm:rank_inequality}, it follows that 
$||F_{B_n}-F_{A_n+pdI}||_\infty=o(1),$ and this completes the proof.
\end{proof}

\begin{remark}
	In \cite{Leibzirer_Ron}, it was remarked that for $np \rightarrow \infty$, $F_{B_n}$ converge in distribution to the standard semi-circle law. This can also be seen easily from the calculations in Theorem \ref{thm:EESD_centred_unsigned}. From Lemma \ref{lemma:eig_unsigned complete}, it follows that $F_{A_n}$ also converges in distribution to the standard semi-circle law as $np \rightarrow \infty$.
\end{remark}

\subsubsection{Proof of Theorem \ref{thm: existence_LSD}: Signed adjacency matrix}

Recall the expectation of signed adjacency matrix $\mathbb{E}A^+$. It is easy to see that $\mathbb{E}A_n^+= p \mathbb{A}_n^+$, where $\mathbb{A}_n^+$ is the signed adjacency matrix of the $d$-dimensional complete simplicial complex  on $n$ vertices. In this part, we prove Theorem \ref{thm: existence_LSD} for the signed case. First, consider the following lemma:
\begin{lemma}[Lemma 8,\cite{uli_wagner}]\label{lemma: eig_A+ complete}
	For $d \geq 2$ and $n \geq d+1$, the eigenvalues of $\mathbb{A}_n^+$ are $n-d$ with multiplicity ${n-1 \choose d-1}$ and $-d$ with multiplicity ${n-1 \choose d}$.
\end{lemma}
\begin{proof}[Proof of Theorem \ref{thm: existence_LSD} (ii)]
Note that by applying Theorem \ref{thm:rank_inequality} to the matrices $A_n^+ +pdI$ and $B_n^+$, we get 
	\begin{equation}\label{eqn:ineqlity_rank}
		||F_{A_n^+ +pdI}-F_{B_n^+}||_\infty \leq \frac{1}{{n \choose d}} \operatorname{Rank}\left(p(\mathbb{A}_n^+ + dI)\right).
	\end{equation}
	It follows from Lemma \ref{lemma: eig_A+ complete} that $\mathbb{A}_n^+ + dI$ has eigenvalue 0 with multiplicity ${n-1 \choose d}=O(n^d)$ and thus $\operatorname{Rank}\left(p(\mathbb{A}_n^+ + dI)\right)=\operatorname{Rank}\left(\mathbb{A}_n^+ + dI\right)= O(n^{d-1})$. Thus from (\ref{eqn:ineqlity_rank}) it follows that $A_n^+ + pdI$ and $B_n^+$ have the same limiting spectral distribution. Since we already know that $A_n^+ + pdI$ and $A_n^+$ have the same limiting spectral distribution, from Theorem \ref{thm: centred LSD}, it follows that
	the limiting spectral distribution of $A_n^+$ is $\Gamma_d(\lambda)$ when $np \rightarrow \lambda$. This proves Theorem \ref{thm: existence_LSD} for the signed adjacency matrix.
\end{proof}

 \subsection{Proofs of Propositions \ref{prop:Gamma unbounded}, \ref{prop:Gamma_semicircle} and \ref{prop:frobenius norm}}
 
  In this section, we derive certain properties of the LSD of adjacency matrix by analysing of moment sequence $\{\beta_h(\lambda)\}_{h \geq 0}$, for both signed and unsigned cases. Using these, we would draw analogies with sparse Erd\H{o}s-R{\'e}nyi random graphs. We first prove Proposition \ref{prop:Gamma unbounded}
\begin{proof}[Proof of Proposition \ref{prop:Gamma unbounded}]
We prove the proposition in two parts: First we show that $\Gamma_d(\lambda)$ is not symmetric around zero and later we show that $\Gamma_d(\lambda)$ is unbounded.
\textcolor{red}{To prove that $\Gamma_d(\lambda)$ is non-symmetric, we show} that $\beta_3(\lambda)\neq 0$ for all $\lambda >0$. We have from (\ref{eqn:k-th moment exp1}) that 
\begin{equation}\label{eqn: 3rd moment nonoriented}
\mathbb{E}\left[\int_{\mathbb{R}} x^{3} \mu_{B_n}(dx)\right]=\frac{1}{N} \sum_{\sigma_{1}, \sigma_{2}, \sigma_{3} \in K^{d-1}} \mathbb{E}\left[B_{\sigma_{1} \sigma_{2}} B_{\sigma_{2} \sigma_{3}}  B_{\sigma_{3} \sigma_{1}}\right].
\end{equation}
It is easy to see that $\mathbb{E}\left[B_{\sigma_{1} \sigma_{2}} B_{\sigma_{2} \sigma_{3}}  B_{\sigma_{3} \sigma_{1}}\right]$ is non-zero only when $\sigma_1 \cup \sigma_2, \sigma_1 \cup \sigma_3$ and $\sigma_2 \cup \sigma_3$ are the same $d-$cell and each of $\sigma_1,\sigma_2$ and $\sigma_3$ are distinct. The number of choices of $d-$cells are ${n \choose d+1}$ and given a $d- $cell, the number of ways of choosing $\sigma_1,\sigma_2$ and $\sigma_3$ is $(d+1) \times d \times (d-1)$. In this case, $\mathbb{E}\left[B_{\sigma_{1} \sigma_{2}} B_{\sigma_{2} \sigma_{3}}  B_{\sigma_{3} \sigma_{1}}\right]= p(1-p)^3 + (1-p)(-p)^3$. Therefore (\ref{eqn: 3rd moment nonoriented}) becomes
\begin{align*}
\mathbb{E}\left[\int_{\mathbb{R}} x^{3} \mu_{B_n}(dx)\right]&= \frac{{n \choose d+1}}{{n \choose d}}\left(p(1-p)^3 + (1-p)(-p)^3\right) \times (d+1) d (d-1)\\
&=d(d-1)(n-d)\left(p(1-p)^3 + (1-p)(-p)^3\right).
\end{align*}
Taking the limit, we get that $\beta_3(\lambda)=d(d-1)\lambda>0$ for all $d \geq 2$ and $\lambda>0$. This implies that $\Gamma_d(\lambda)$ is non-symmetric for all $\lambda>0$.

To prove that $\Gamma_d(\lambda)$ has unbounded support, it is sufficient show that $\beta_{4k}^{1/k}$ converges to infinity as $k \rightarrow \infty$. Consider the subset $B_k \subseteq \widetilde{W}_{k+d}^{4d}$ defined as the set of all $w=\sigma_{1}\sigma_{2}\cdots \sigma_{4k+1}$ with following properties:
\begin{enumerate}[(i)]
	\item $\sigma_1=\sigma_3=\sigma_5=\cdots=\sigma_{4k+1}$,
	\item $\operatorname{supp}_d(w)=\{\sigma_{2i-1}\cup \sigma_{2i}:1\leq i \leq k \}$,
	\item for $1 \leq i \neq j \leq k$, the $d$-cells $\sigma_{2i-1}\cup \sigma_{2i}$ and $\sigma_{2j-1}\cup \sigma_{2j}$ are distinct.
\end{enumerate}
We count the number of elements in $B_k$. Note that the representative element for $\sigma_{1}$ is always $\{1,2,\ldots,d\}$ and therefore by condition (i), the number of choices for $\sigma_i$ is always equal to 1 for all odd values of $i$. Now we turn our attention to $i=2,4,\ldots,4k$. By condition (iii), for each $1 \leq j \leq k$, the $d$-cell $\tau=\sigma_{2j-1}\cup \sigma_{2j}$ is the first appearance of $\tau$. Hence the new $0$-cell added to $\tau$ is fixed, as it is in the ascending order. Therefore, the number of choices for $\sigma_{2j}$ is determined entirely by the value of $r_{2j-1}$, the first component of the labelling on the edge $\{2j-1,2j\} \in E(G_w)$. Thus, the number of choices for $\sigma_{2j}$, in this case, is $d$. For $j$ such that $k+1 \leq j \leq 2k$, note that the $d$-cell $\tau=\sigma_{2j-1}\cup \sigma_{2j}$ could be any of the $d$-cells in $\operatorname{supp}_d(w)$ as all the $d$-cells have already occurred. Thus the total number of choices for $\sigma_{2j}$ in this case is $k \times d=kd$. Hence $	|B_k|= d^k \times (kd)^k$. Therefore, it follows that
\begin{equation*}
	\beta_{4k}^{1/k} \geq \left(|\widetilde{W}_{k+d}^{4k}|\lambda^k\right)^{1/k} \geq \left(|B_k|\lambda^k\right)^{1/k}=\left(d^k(kd)^k\lambda^k\right)^{1/k}=d^2k\lambda.
\end{equation*}
As $d^2k\lambda$ converges to infinity for all choices of $\lambda > 0$, it follows that $\beta_{4k}^{1/k}$ converges to infinity as $k \rightarrow \infty$. 
\end{proof}

The following lemma from \cite{antti_ron} helps in proving Proposition \ref{prop:Gamma_semicircle}.
\begin{lemma}[Lemma 3.11, \cite{antti_ron}]\label{Lemma: W^k_k/2+d} 
The following claims hold for even $k$ and $w \in \mathcal{W}_{k / 2+d}^{k} .$
\begin{enumerate}[(i)]
\item $N_{w}(\tau) \in\{0,2\}$ for every $d$-cell $\tau$.
\item  $\left|\mathcal{W}_{k / 2+d}^{k}\right|=\mathcal{C}_{\frac{k}{2}} d^{k / 2}$,
\end{enumerate}
where $\mathcal{C}_k$ is the $k-$th Catalan number and is given by $\mathcal{C}_k= \frac{1}{k+1}{2k \choose k}$.
\end{lemma}

\begin{proof}[Proof of Proposition \ref{prop:Gamma_semicircle}]
For fixed $d \geq 2$ and $\lambda>0$, $k-$th moment of $\frac{1}{\sqrt{\lambda d}}\Gamma_d(\lambda)$ is given by
\begin{align*}
\displaystyle\frac{\beta_{k}(\lambda)}{\left(\lambda d\right)^{k/2}}=\frac{1}{d^{k/2}}\sum_{s=d+1}^{\lfloor\frac{k}{2}\rfloor +d}\left|\widetilde{W}_{s}^{k}\right| \lambda^{s-d-k/2}.
\end{align*}
Due to method of moments, it is sufficient to show that for each $k \in \N$, $\frac{\beta_{k}(\lambda)}{\left(\lambda d\right)^{k/2}}$ converges to the $k$-th moment of standard semi-circle law as $\lambda \rightarrow \infty$.

Consider a fixed $k \in \mathbb{N}$. Note that from Lemma \ref{claim:size W_k,s}, $|\widetilde{\mathcal{W}}_s^k|$ is bounded above by $((k+d)d)^k$, a constant.
This gives that for all values of $s<\frac{k}{2}+d$, $|\widetilde{\mathcal{W}}_s^k|\lambda^{s-d-k/2}$ is of the order $O(\lambda^{-\frac{1}{2}})$. Therefore 
 \begin{equation*}
 	\lim_{\lambda \rightarrow \infty} \displaystyle\frac{\beta_{k}(\lambda)}{\left(\lambda d\right)^{k/2}}=
 	\begin{cases}
 		\frac{1}{d^{k/2}} |\widetilde{W}_{\frac{k}{2}+d}^{k}| &\text{ for } k \text{ even},\\
 		0 &\text{ for } k \text{ odd.}
 	\end{cases}
 \end{equation*} 

By Lemma \ref{Lemma: W^k_k/2+d} $(i)$, it follows that $|\operatorname{supp}_d(w)|= k/2$ for all even $k$ and words $w \in \mathcal{W}^k_{k/2+d}$. Thus $\widetilde{W}_{k/2+d}^k=\mathcal{W}^k_{k/2+d}$ for every even $k$. Now using $(ii)$ of Lemma \ref{Lemma: W^k_k/2+d}, we obtain that for each even $k$,
\begin{equation*}
\lim_{n \rightarrow \infty} \frac{\beta_{k}(\lambda)}{(\lambda d)^{k/2}} = \mathcal{C}_{\frac{k}{2}}.
\end{equation*}
Hence from  method of moments, it follows that $\frac{1}{\sqrt{\lambda d}}\Gamma_d(\lambda)$ converges in distribution to the standard semi-circle law as $\lambda \rightarrow \infty$.
\end{proof}
%
Now, the only part left in this subsection is to prove Proposition \ref{prop:frobenius norm}.
\begin{proof}[Proof of Proposition \ref{prop:frobenius norm}]
We first compute $\beta_2(\lambda)$ for $\lambda>0$. For any equivalence class in $\widetilde{\mathcal{W}}_{d+1}^2$, as the representative of the first $(d-1)$-cell, $\sigma_1$ is always $\{1,2,\ldots ,d\}$. Therefore, the number of choices for $\sigma_1$ is exactly 1. Note that $\sigma_2$ is obtained from the removal of a $0$-cell from $\sigma_1$ followed by the addition of the $0$-cell $\{d+1\}$. Therefore the number for choices of $\sigma_2$ is $d$. Thus for any $\lambda>0$, $\beta_2(\lambda)=d$.
	
	Note that as $B_n$ and $B_n^+$ are symmetric, thus $B_nB_n^t=B_n^2$ and $B_n^+B_n^{+^t}=(B_n^+)^2$ and therefore $\Tr(H_nH_n^t)= \Tr(H_n^+H_n^{+^t})$ and from the proof of Theorem \ref{thm:EESD_centred_unsigned} it follows that $||H_n||_F^2$ and $||H_n^+||_F^2$ converge almost surely to $\beta_2 \lambda= d\lambda$, implying the result.
\end{proof}
\section{Local Weak Convergence}\label{sec: local_weak}

The principal objective of this section is to prove Theorem \ref{thm: weak limit} and Theorem \ref{thm: spectral measure dGW}. The proofs here are more analytical in nature and would use ideas from local weak convergence. We first introduce the necessary terminology required to prove Theorems \ref{thm: weak limit} and \ref{thm: spectral measure dGW}. 
\subsection{ Random Graphs associated with $Y_d(n,p)$}

Recall the definition of unsigned adjacency matrix from Definition \ref{defn:adjacency matrix}. In this section, our aim is to introduce two graphs associated to the unsigned adjacency matrix of the Linial-Meshulam complex. As in Section \ref{sec: Existence LSD}, we shall use $A_n$ to denote the unsigned adjacency matrix of the Linial-Meshulam complex   $Y_d(n,p)$. We define $U_n= K^d(n)$ and $V_n=K^{d-1}(n)$, where $K^j(n)$ is the set of all $j$-cells on $n$ vertices.

First, we define a graph with same adjacency matrix as the unsigned adjacency matrix of the Linial-Meshulam complex $Y_d(n,p)$. This graph is called the line graph of the simplicial complex and is a subgraph of the intersection graph on the family of all $(d-1)$ cells.
\begin{definition}\label{defn: graph Gn}
For $d \geq 2, n \geq d+1$ and $0 < p < 1$,  consider the random simplicial complex $Y_d(n,p)$ with unsigned adjacency matrix $A_n$. The line graph of $Y_d(n,p)$ is defined as the graph $G_n =(V_n, E_n)$ where $V_n=K^{d-1}(n)$ and for $v, v^{\prime} \in V_n$, $\{v,v^{\prime}\} \in E_n$ if $(A_n)_{vv^{\prime}}=1$. Notice that adjacency matrix of the line graph of $Y_d(n,p)$ is $A_n$, the unsigned adjacency matrix of $Y_d(n,p)$. 
\end{definition} 
Line graphs of hypergraphs is a popular area of study in combinatorics. For an in-depth study of line graphs and their applications, see \cite{SIAM_intersection_graph}.
Theorem \ref{thm: existence_LSD} states that the limiting spectral distributions of the unsigned and signed adjacency matrices of $Y_d(n,p)$ are reflections of each other. Thus, to study the limiting spectral distribution
of adjacency matrices, it is sufficient to study the line graphs of $Y_d(n,p)$. This allows us to use the analytical techniques associated with rooted graphs to study the properties of LSD.

A bipartite graph is an undirected graph $G=(U,V,E)$ with vertex set $ U \cup V$ and edge set $E$ such that $U$ and $V$ are disjoint sets, and $\{u,v\} \in E$ only if $u \in U$ and $v \in V$. 
The concept of isomorphism on finite graphs naturally carries over to bipartite graphs. We also maintain the convention that for a bipartite graph the root $o$ always belongs to $V$. An equivalence class of rooted locally finite bipartite graphs is called an unlabelled bipartite graph. We shall use $\widehat{\mathcal{G}}^*$ to denote the space of all unlabelled rooted bipartite graphs.

The following bipartite graph was introduced in \cite{linial_peled} for studying the Laplacian of Linial-Meshulam complex.
\begin{definition}
For  $d \geq 2, n \geq d+1$ and $0 <p<1$, let $Y_d(n,p)$ denote the $d$-dimensional random simplicial complex. We define the bipartite graph $\widehat{G}_n=(U_n,V_n, \widehat{E}_n)$ where $U_n= K^{d}(n), V_n =K^{d-1}(n)$ and for $u \in U_n$ and $v \in V_n$, $\{u,v\} \in \widehat{E}_n^{(X)}$ if $u \in Y_d(n,p)$ and $v \subset u$.
\end{definition}
Throughout the rest of the paper, $\widehat{G}_n$ shall denote this particular graph, $\widehat{E}_n$ its edge set and $\widehat{A}_n$, its adjacency matrix. Notice that for $u \in U_n$ and $v \in V_n$, $\{u,v\} \in \widehat{E}_n$ if and only if there exists $v^\prime \in V_n$ such that $v \cup v^\prime=u$ and $(A_n)_{vv^\prime}=1$. The measure induced on $\widehat{\mathcal{G}}^*$ by $\widehat{G}_n$ where the root is chosen uniformly from $V_n$ will be denoted by $\nu_{d,n}$. As $\widehat{G}_n$ is a random graph, $\nu_{d,n}$ is a random measure.

The following map  $\varphi$ from the class of labelled rooted bipartite graphs to the class of labelled rooted graphs connects the two types of graphs. For a bipartite graph $G=(U,V,E,o)$ with $o \in V$, we define $\varphi(G)$ as the rooted graph with the vertex set $V$, root $o$ and for $v_1 \neq v_2$, $\{v_1,v_2\}$ belong to the edge set of $\varphi(G)$ if 
there exist $u \in U$ such that $\{v_1,u\} \in E$ and $\{v_2,u\} \in E$.
Using $\varphi$ we define a map $\hat{\varphi} : \widehat{\mathcal{G}}^* \rightarrow \mathcal{G}^*$, as
\begin{equation}\label{eqn:phi hat}
\hat{\varphi}\left[(U,V,E,o)\right]= [\varphi(U,V,E,o)],
\end{equation}
where $\left[(U,V,E,o)\right]$ is the equivalence class of $(U,V,E,o)$ in $\widehat{\mathcal{G}}^*$ and $\left[\varphi(U,V,E,o)\right]$ is the equivalence class of $\varphi(U,V,E,o)$ in $\mathcal{G}^*$. It is not hard to see that $\hat{\varphi}$ is well-defined. The following proposition says $\hat{\varphi}$ is continuous.

\begin{proposition}\label{prop: phi hat cont}
The map $\hat{\varphi}$ is continuous on $\widehat{\mathcal{G}}^*$.
\end{proposition}
\begin{proof}
Since $\widehat{\mathcal{G}}^* \subseteq \mathcal{G}^*$, (\ref{eqn: metric G*}) induces a metric on $\widehat{\mathcal{G}}^*$.

Let $g,h \in \widehat{\mathcal{G}}^* $ and suppose $(g)_t \simeq (h)_t$ for some $t \geq 0$. It follows from the well-definedness of $\hat{\varphi}$, that $\varphi\left((g)_t\right) \simeq \varphi\left((h)_t\right)$. Furthermore, note that if $v,v^\prime \in V$ are at graph distance $2r$ in $g$, then the distance between $v$ and $v^\prime$ is $r$ in $\varphi(g)$. Thus $\left(\varphi(g)\right)_{\frac{t}{2}} \simeq \left(\varphi(h)\right)_{\frac{t}{2}}$, by the definition of $\varphi$. Therefore 
\begin{equation}\label{eqn: T(G,G')}
T_{\varphi(g),\varphi(h)} \geq \frac{T_{g,h}-1}{2},
\end{equation} where $T_{g,h}$ is as defined in (\ref{eqn: metric G*}). From (\ref{eqn: T(G,G')}), it follows that 
\begin{equation*}
d_{\mathcal{G^*}}\left(\varphi(g),\varphi(h)\right) = \frac{1}{1+T_{\varphi(g),\varphi(h)}} \leq \frac{1}{1+\frac{T_{g,h}-1}{2}} = \frac{2}{1+ T_{g,h}}= 2\times d_{\mathcal{G^*}}\left(g,h\right).
\end{equation*}
Hence, the map $\hat{\varphi}$ is continuous on $\widehat{\mathcal{G}}^*$.
\end{proof}

\subsection{d-block Galton-Watson graphs}

In this subsection, we introduce a special type of random graph known as $d$-block Galton-Watson graph.  $d$-block Galton-Watson graphs are generalizations of Galton-Watson trees in the sense that the $1$-block Galton-Watson graph is a Galton-Watson tree. For $d \geq 2$, all realizations of  $d$-block Galton-Watson graph except the trivial realization are non-trees (See Figure \ref{fig: d-block graph}).

First, we introduce an explicit construction of Galton-Watson trees.
Define $\N^f := \cup_{k \geq 0} \N^k$, with the convention $\N^{0}=\phi$ and $\N=\{1,2,3,\ldots\}$. For a sequence $(N_{\mathbf{i}}), \mathbf{i} \in \N^f$ of non-negative integers, we define the set 
\begin{equation}\label{eqn: defn V}
	V = \{\mathbf{i}=(i_1,i_2,\ldots i_k) \in \N^f \text{ and } \forall \,\, 1 \leq \ell \leq k , 1 \leq i_{\ell} \leq N_{(i_1,i_2,\ldots, i_{\ell-1})} \}.
\end{equation}
For $\mathbf{i} \in V$, we call the elements of the sets $\{ (\mathbf{i},1),(\mathbf{i},2), \ldots , (\mathbf{i},N_{\mathbf{i}})\}$ as children of $\mathbf{i}$. For $k \geq 2$, $\mathbf{i}=(i_1,i_2,\ldots , i_k) \in \N^k$, we define $(i_1,i_2,\ldots , i_{k-1})$ as the ancestor of $\mathbf{i}$ and for $\mathbf{i} \in \N$, $\phi$ is defined as the ancestor of $\mathbf{i}$. We define a rooted tree $T=(V,E,o\phi)$ by putting an edge between all vertices in $V$ and its children. For the rooted tree with vertex set $V$ as in (\ref{eqn: defn V}), we define the depth of the root as $0$, and for other vertices depth is defined in an iterative fashion as, if $u$ is a child of $v$, then $\operatorname{depth}(u):= \operatorname{depth}(v)+1$.
\begin{definition}
	Let $P$ be a probability distribution on the set of non-negative integers and let $V$ be as defined in (\ref{eqn: defn V}). If $\{N_{\mathbf{i}}\}$ is an i.i.d. sequence with distribution $P$ then the random tree $T$ is called a Galton-Watson tree with offspring distribution $P$.
\end{definition}
Note that by its definition, Galton-Watson tree is a random rooted graph and therefore gives a probability measure on $\mathcal{G^*}$. Galton-Watson trees and its modifications come up as local weak limits for different random structures . Galton-Watson trees are in general not unimodular. In fact, a Galton-Watson tree is unimodular if and only if $P$ is a Poisson distribution \cite{coursSRG}. 

Next, we define the concept of $d$-siblings. Consider a random rooted tree $T=(V,E,o)$ where $V$ is as in (\ref{eqn: defn V}) and offspring distribution $\{N_{\mathbf{i}}\}$ such that $\{N_{\mathbf{i}}\}$ are i.i.d random variables with distribution $d \times P$, where $P$ is a probability distribution on non-negative integers. For the random rooted tree $T$ and $\mathbf{i} \in V$, we say that the vertices $(\mathbf{i},s), (\mathbf{i},r)$ such that $1 \leq r\neq s \leq N_{\mathbf{i}}$ are $d-$siblings if $\lfloor \frac{s-1}{d} \rfloor= \lfloor \frac{r-1}{d} \rfloor$. Now, we define $d$-block Galton-Watson graphs.

\begin{definition}
	Let $P$ be a probability distribution on non-negative integers. For $d \in \N$, a $d$-block Galton-Watson graph with offspring distribution $dP$ is defined as the graph $T=(V,E_d,o)$ such that $N_{\mathbf{i}}$ is an i.i.d. sequence with distribution $dP$, $V$ is as in (\ref{eqn: defn V}), and edges are put between all vertices of $V$ and its ancestor and $d-$siblings. The measure induced on $\mathcal{G}^*$ by $d$-block Galton-Watson graphs with offspring distribution $dP$ shall be denoted by $dGW(dP)$.
\end{definition}

To better understand $d$-block Galton-Watson graphs, we recall the notion of biconnectedness. A graph is called biconnected if it is connected, and remains connected even after the removal of any vertex and the edges incident to it. A maximal biconnected subgraph of a graph is called a biconnected component. 
\begin{definition}
 A complete subgraph of a graph is called a clique. A graph $G$ is called a block graph if every biconnected component of $G$ is a clique. A block graph with every biconnected component as a clique of size $d+1$ is called a $d$-block graph.
\end{definition}

Clearly, each realization of $d$-block Galton-Watson graph with more than  one vertex is a $d$-block graph. Since every block graph is a geodetic graph, it follows that there exists a unique shortest path between any two vertices of a $d$-block Galton-Watson graph. For a non-root vertex $v \in V$, let $(v=v_0, v_2, \ldots , v_k=o)$ be the shortest path from $v$ to the root $o$. Then we define $v_1$ as the ancestor of $v$, $v$ as the child of $v_1$, and $\operatorname{depth}(v)=k-1$. We define $\operatorname{depth}(o)=0$.

For an unlabelled rooted graph $G=(V,E,o)$, and a vertex $v \neq o$, we define the graph rooted at $v$ as $H=(V_H, E_H,v)$ as the subgraph induced on the vertex set 
$$V_H=\{u \in V: \text{ the shortest path from $u$ to $o$, passes through $v$}\},$$
with the root $v$.
For a $d$-block Galton-Watson graph with offspring distribution $dP$, the graph rooted on $v$ is again a $d$-block Galton-Watson graph with offspring distribution $dP$ for any vertex $v$. We use this fact to prove proposition \ref{prof: dGW unimodular}. 

Now, we proceed to define the concept of unimodularity of graphs. We define a graph with two roots as a triple $(G,o,o^{\prime})$ where $G$ is a graph with two distinguished vertices $o$ and $o^\prime$. For graphs with two roots, $(G_1,o_1,o_1^\prime)$ and $(G_2,o_2,o_2^\prime)$ are defined to be isomorphic if   there exists $\sigma:V_1 \rightarrow V_2$ such that $(G_1,o_1) \simeq (G_2,o_2)$ under the isomorphism $\sigma$ and $\sigma(o_1^\prime)=o_2^\prime$. We define $\mathcal{G}^{**}$ as the set of all equivalence classes of locally finite connected graphs with two roots.
\begin{definition}
	A probability measure $\rho \in \mathcal{P}\left(\mathcal{G}^* \right)$ is called unimodular if for any measurable function $f:\mathcal{G}^{**}\rightarrow \mathbb{R}_+$,
	\begin{equation}\label{defn:unimodularity}
		\displaystyle\int_{\mathcal{G}^*}\sum_{v \in V}f\left( G,o,v \right) d\rho\left(G,o\right)= \int_{\mathcal{G}^*}\sum_{v \in V}f\left( G,v,o \right) d\rho\left(G,o\right).
	\end{equation}
\end{definition} 

It is quite convenient to reduce (\ref{defn:unimodularity}) to functions $f$ such that $f(G,u,v)=0$ if $\{u,v\}\notin E(G)$.

\begin{proposition}[Proposition 2.2, \cite{Lyon_Aldous}]\label{prop: involution invar}
	Let $\rho \in \mathcal{P}(\mathcal{G}^*)$. Then $\rho$ is unimodular if and only if (\ref{defn:unimodularity}) holds for all measurable functions $f:\mathcal{G}^{**} \rightarrow \mathbb{R}_+ $ such that $f(G,u,v)=0$ if $\{u,v\}\notin E(G)$.
\end{proposition}

It is easy to see that for any finite graph $G$, the measure $U(G)$ is unimodular. In fact, it is known that the local weak limit of a sequence of graphs is always unimodular\cite{Lyon_Aldous}. We denote the set of all unimodular measures on $\mathcal{G}^*$ by $\mathcal{P}_{\operatorname{uni}}\left(\mathcal{G}^* \right)$. It is known that
$\mathcal{P}_{\operatorname{uni}}\left(\mathcal{G}^* \right)$ is closed under local weak topology. i.e. if $(\rho_n)_{n \in \N} \subset \mathcal{P}_{\operatorname{uni}}\left(\mathcal{G}^* \right)$ and $\rho_n \rightarrow \rho$ weakly, then $\rho \in \mathcal{P}_{\operatorname{uni}}\left(\mathcal{G}^* \right)$ \cite{Benjamin_Schramm_lwc}. Thus, for Theorem \ref{thm: weak limit} to be true, it is necessary that $dGW(d\operatorname{Poi}(\lambda))$ is unimodular measure for all $\lambda>0$. The next proposition gives an independent conformation of this fact.
\begin{proposition}\label{prof: dGW unimodular}
Let $P$ be a probability measure on $\mathbb{N}$. The measure $dGW(dP)$ is unimodular if and only if $P$ is a Poisson random variable.
\end{proposition}
\begin{proof}
We first prove that $dGW(d\operatorname{Poi}(\lambda))$ is unimodular. Let $G=(V,E)$ be the $d$-block Galton-Watson graph with offspring distribution $dP$, where $P=\operatorname{Poi}(\lambda)$ for some $\lambda>0$. To prove unimodularity, by Proposition \ref{prop: involution invar}, it is sufficient to prove (\ref{defn:unimodularity}) for functions $f$ such that $f(G,u,v)=0$ if $\{u,v\} \notin E(G)$.

Let $N_o$ denote the number of offsprings of the root $o$. Then
\begin{align*}
\mathbb{E}\left(\sum_{i=1}^{N_o}f(G,o,i)\right)&=\sum_{k=1}^{\infty}P(k)\mathbb{E}\left(\sum_{i=1}^{kd}f(G,o,i)\big|N_o=kd\right)\\
&=\sum_{k=1}^{\infty}kdP(k)\mathbb{E}\left(f(G,o,1)\big|N_o=kd\right).
\end{align*}
Note that since $P$ is a Poisson distribution, $P(k-1)=P(k)/\mathbb{E}[P]$. Additionally, let $N$ denote a random variable independent of $\{N_\mathbf{i}\}$ and with distribution $dP$. Then since $\mathbb{E}[N]=d\mathbb{E}[P]$, we get
\begin{align}\label{eqn: E unimodular}
\mathbb{E}\left(\sum_{i=1}^{N_o}f(G,o,i)\right)&=\sum_{k=1}^{\infty}\mathbb{E}[N]P(k-1)\mathbb{E}\left(f(G,o,1)\big|N_o=kd\right)\nonumber\\
&=\mathbb{E}[N]\sum_{k=0}^{\infty}P(k)\mathbb{E}\left(f(G,o,1)\big|N_o=(k+1)d\right).
\end{align}
For $d$-block Galton-Watson graphs $G_1,G_2,\ldots, G_{kd}$ with offspring distribution $dP$, we define the random graph $G=R_{u,v}(G_1,G_2,\ldots, G_{kd})$ as the $d$-block Galton-Watson graph such that $u \in V_G$ has $kd$ neighbours numbered $v=v_1,v_2,\ldots , v_{kd}$ with the subgraph rooted at $v$ isomorphic to $G_1$ and the subgraph rooted at $v_r$ isomorphic to $G_r$. Further, an edge exists between $v_i$ and $v_j$ if and only if $\lfloor \frac{i-1}{d} \rfloor= \lfloor \frac{j-1}{d} \rfloor$ (See Figure \ref{fig: R_u,v graph}).

\begin{figure}[h]
\includegraphics[height=60mm, width =150mm]{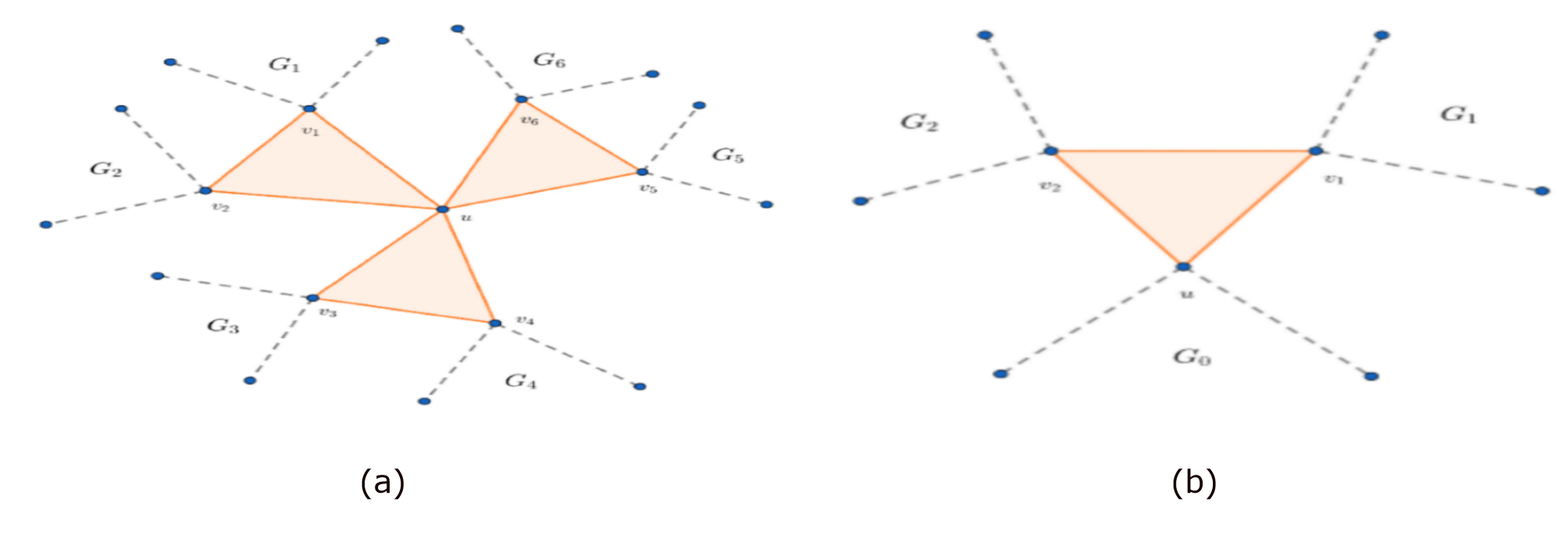}
	\caption{(a) The graph $R_{u,v_1}(G_1,G_2,\ldots, G_6)$ constructed from $2$-block Galton-Watson graphs $G_1,G_2,\ldots , G_6$. (b) The graph $S_{u,v_1}(G_0,G_1,G_2)$ constructed from $2$-block Galton-Watson graphs $G_0, G_1,G_2$.}
	\label{fig: R_u,v graph}
\end{figure}
Given $N_o=(k+1)d$, $[(G,o,1)]$ has the same law as $\left[R_{o,1}\left(G_1,G_2,\ldots, G_{(k+1)d}\right)\right]$. Thus from (\ref{eqn: E unimodular}), we get
\begin{align*}
\mathbb{E}\left(\sum_{i=1}^{N_o}f(G,o,i)\right)&=\mathbb{E}[N]\sum_{k=0}^{\infty}P(k)\mathbb{E}\left(f\left(R_{o,1}\left(G_1,G_2,\ldots, G_{(k+1)d}\right),o,1\right)\big|N_o=(k+1)d\right)\\
&=\mathbb{E}[N]\mathbb{E}\left(f\left(R_{o,1}\left(G_1,G_2,\ldots, G_{(\hat{N}+1)d}\right),o,1\right)\right),
\end{align*}
where $\hat{N}$ is independent of $N$ and has distribution $\operatorname{Poi}(\lambda)$.

Now, for $d$-block Galton-Watson graphs $G_0, G_1,G_2,\ldots, G_{d}$ with offspring distribution $dP$, we define the random graph $G=S_{u,v}(G_0, G_1, G_2,\ldots, G_{d})$ as the $d$-block Galton-Watson graph such that $u \in V_G$ has $d$ neighbours numbered $v=v_1,v_2,\ldots , v_{d}$ with all of them together forming a clique of size $d+1$. Furthermore, the subgraph rooted at $u$ is isomorphic to $G_0$ and subgraph rooted  at $v_r$ is isomorphic to $G_r$ for all $r$. 

Note that the root $o$ has $(\hat{N}+1)d$ neighbours in the graph $R_{o,1}(G_1,G_2,\ldots, G_{(\hat{N}+1)d})$. Let $1,2,\ldots , d$ denote neighbours of $o$ in one of the cliques containing $o$. The rest of the neighbours of $o$ can be treated as a $d$-block Galton-Watson graph with offspring distribution $d\operatorname{Poi}(\lambda)$ rooted at $o$. This implies that $\left(R_{o,1}(G_1,G_2,\ldots, G_{(\hat{N}+1)d}),o,1\right)$ has same law as $\left(S_{o,1}(\hat{G}_0,\hat{G}_1,\hat{G}_2,\ldots, \hat{G}_{d}),o,1\right)$, where $\hat{G}_r$'s are $d$-block Galton-Watson graphs with offspring distribution $d\operatorname{Poi}(\lambda)$. Thus 
\begin{align*}
\mathbb{E}[N]\mathbb{E}\left(f\left(R_{o,1}\left(G_1,G_2,\ldots, G_{(\hat{N}+1)d}\right),o,1\right)\right)&=\mathbb{E}[N]\mathbb{E}\left(f\left(S_{o,1}(\hat{G}_0,\hat{G}_1,\hat{G}_2,\ldots, \hat{G}_{d})\right),1, o\right)\\
&=\mathbb{E}[N]\mathbb{E}\left(f\left(S_{o,1}(\hat{G}_0,\hat{G}_1,\hat{G}_2,\ldots, \hat{G}_{d})\right),1,o\right),
\end{align*}
where the last equality follows from the observation that $\left(S_{u,v}(\hat{G}_0,\hat{G}_1,\hat{G}_2,\ldots, \hat{G}_{d}),u,v \right)$ has the same law as $\left(S_{u,v}(\hat{G}_0,\hat{G}_1,\hat{G}_2,\ldots, \hat{G}_{d}),v,u \right)$. 

A similar computation shows that 
\begin{align*}
\mathbb{E}\left(\sum_{i=1}^{N_o}f(G,i,o)\right)&=\mathbb{E}[N]\mathbb{E}\left(f\left(R_{o,1}\left(G_1,G_2,\ldots, G_{(\hat{N}+1)d}\right),o,1\right)\right)\\
&=\mathbb{E}[N]\mathbb{E}\left(f\left(S_{o,1}(\hat{G}_1,\hat{G}_2,\ldots, \hat{G}_{d})\right),1,o\right),
\end{align*}
which proves that $dGW(\operatorname{Poi}(\lambda))$ is a unimodular measure.

To show that every unimodular $d$-block Galton-Watson graph has offspring distribution $d\operatorname{Poi}(\lambda)$, consider a unimodular $d$-block Galton-Watson graph $G$ with offspring distribution $dP$. Consider $f(G,u,v)=\mathbf{1}(\operatorname{deg}(u)=kd)$. Then
\begin{align*}
\mathbb{E}\left(\sum_{v \in V}f(G,o,v)\right) &= kdP(k) \,\,\,\text{and}\\
\mathbb{E}\left(\sum_{v \in V}f(G,v,o)\right)&= \sum_{r=1}^{\infty}\mathbb{P}(N_o=dr)\mathbb{E}\left(\sum_{v \in V}\mathbf{1}(\operatorname{deg}(u)=kd)\Big|N_o=dr\right)\\
&=\sum_{r=1}^{\infty} P(r) \times dr \times P(k-1).
\end{align*}
Equating both the terms, we get the required result.
\end{proof}
\subsection{Proof of Theorem \ref{thm: weak limit}}

In this subsection, we prove Theorem \ref{thm: weak limit}. We first define a random tree known as Poisson $d$-tree which was earlier used in \cite{linial_peled}. Using the concept of push-forward measure, we connect Poisson $d$-trees to $d$-block Galton-Watson graphs and later use it to prove Theorem \ref{thm: weak limit}.

\begin{definition}\label{defn: Poisson d-tree}
For $d \geq 2$, a $d-$tree is a rooted tree in which each vertex at odd depth has exactly $d$ children. For $d \geq 2$ and $\lambda >0$, a Poisson d-tree with parameter $\lambda$, is the random d-tree in which the number of children of every vertex at even depth are independent random variables
with distribution $\operatorname{Poi}(\lambda)$. The measure induced by Poisson $d$-tree with parameter $\lambda$ on $\widehat{\mathcal{G}}^*$ will be denoted by $\nu_{d,\lambda}$.
\end{definition}

\begin{definition}
Given measurable spaces $(X_1,\tau_1)$, $(X_2,\tau_2)$ and a measurable map $F: X_1 \rightarrow X_2$, the push-forward measure of a measure $\mu$ on $(X_1,\tau_1)$ is the measure on $(X_2,\tau_2)$ defined by
\begin{equation}\label{eqn:pushforward measure defn}
\mu^F(B):=\mu\left(F^{-1}(B)\right), \,\, \forall \, B \in \tau_2.
\end{equation}
\end{definition}

Using the change of variable formula for push-forward measure (Theorem 3.6.1, \cite{Bogachev}), we have that for integrable functions $g: X_2 \rightarrow \mathbb{R}$,
\begin{equation}\label{eqn:change_of_var}
\int_{X_{2}} g \mathrm{~d}\mu^F=\int_{X_{1}} g \circ F \mathrm{~d} \mu.
\end{equation}
Recall the definition of the map $\hat{\varphi}$ in (\ref{eqn:phi hat}). Since the map $\hat{\varphi}$ is continuous, it is measurable with respect to the Borel $\sigma-$algebra on $\widehat{\mathcal{G}}^*$. 
Now, we state two lemmas required in the proof of Theorem \ref{thm: weak limit}.

\begin{lemma}\label{lemma: hat phi bijection}
For $d \geq 2$ and $\lambda >0$, the map $\hat{\varphi}:\{$set of all realizations of unlabelled Poisson $d$-trees with parameter $\lambda\} \rightarrow \{$set of all realizations of unlabelled $dGW$ graphs with offspring distribution $d\operatorname{Poi}(\lambda)\, \}$ is a bijection. As a consequence, $\nu_{d,\lambda}^{\hat{\varphi}}=dGW(d\operatorname{Poi}(\lambda))$ for all $\lambda > 0$.
\end{lemma}
\begin{lemma}\label{lemma: U(Gn) convergence}
	For $d \geq 2$ and $p(n)= \lambda/n$ for some $\lambda > 0$, $U(G_n)$ converge weakly to $\nu_{d,\lambda}^{\hat{\varphi}}$ almost surely, where $\nu_{d,\lambda}^{\hat{\varphi}}$ is the pushforward measure of $\nu_{d,\lambda}$.
\end{lemma}
\begin{proof}[Proof of Theorem \ref{thm: weak limit}]
Now, we prove Theorem \ref{thm: weak limit} assuming Lemmas \ref{lemma: hat phi bijection} and \ref{lemma: U(Gn) convergence}. Consider $\lambda>0$ and $np \rightarrow \lambda$. Note that from Lemma \ref{lemma: hat phi bijection}, it follows that$\nu_{d,\lambda}^{\hat{\varphi}}=dGW(d\operatorname{Poi}(\lambda))$ and applying Lemma \ref{lemma: U(Gn) convergence}, we get that $U(G_n)$ converges weakly to $dGW(d\operatorname{Poi}(\lambda))$ almost surely.
\end{proof}
Now, we proceed to prove Lemma \ref{lemma: hat phi bijection} and Lemma \ref{lemma: U(Gn) convergence}. First we prove Lemma \ref{lemma: hat phi bijection}
\begin{proof}[Proof of Lemma \ref{lemma: hat phi bijection}]
For a tree $T$, a vertex $v$ is called a grandchild of another vertex $v^{\prime}$, if there exists $u$ such that $u$ is a child of $v^{\prime}$ and $v$ is a child of $u$. Let $T=(U,V,E,o)$ be a realization of the Poisson $d-$tree with parameter $\lambda$, where $U$ is the set of all vertices at odd depth and $V$ is the set of all vertices at even depth.

By the definition of Poisson $d-$tree, each $v \in V$ has offspring distribution $\operatorname{Poi}(\lambda)$, and each of its children has exactly $d$ offsprings. Consider $v \in V$, with children labelled as $u_1,u_2,\ldots,u_k$, and grandchildren labelled as $v_1,v_2,\ldots,v_{kd}$ where for each $0 \leq i \leq k-1$ and $1 \leq r,s \leq d$, $v_{i+r}$ and $v_{i+s}$ are the offspring of $u_i$. Therefore, it follows from the definition of $\hat{\varphi}$ that the edge $\{v_i,v_j\} \in E(\hat{\varphi}(T))$ if and only if $\lfloor \frac{i-1}{d} \rfloor= \lfloor \frac{j-1}{d} \rfloor$. Furthermore, $\{v,v_i\} \in E(\hat{\varphi}(T))$ for all $1 \leq i \leq kd$. Hence, $\hat{\varphi}(T)$ is a realization of the $d$-block Galton-Watson graph with offspring distribution $d\operatorname{Poi}(\lambda)$ and thus the map $\hat{\varphi}:\{$set of all realizations of unlabelled Poisson $d$-trees with parameter $\lambda\} \rightarrow \{$set of all realizations of unlabelled $dGW$ graphs with offspring distribution $d \operatorname{Poi}(\lambda)\, \}$ is onto.

To prove the bijection, consider a realization $G$ of the $d$-block Galton-Waton graph with vertex set $V$ and root $o$. Construct the set 
\begin{equation*}
	U=\big\{(v,v_1,\ldots,v_d): v_i \text{ is a child of } v \text{ and for all } 1 \leq i\neq j \leq d, \{v_i,v_j\} \in E(G)\big\}.
\end{equation*}
Consider the rooted graph $T$ with vertex set $U \cup V$, root $o$ and for $v \in V$ and $u \in U$, $\{u,v\} \in E(T)$ if $v$ is a component of $u$. For $v \in V, u \in U$ such that $\{v,u\} \in E(T)$, we say that $u$ is a child of $v$, if $v$ is the first component of $u$. Otherwise we say that $v$ is a child of $u$. Note that $T$ is a bipartite tree such that the set of vertices at odd depth is $U$ and the set of all vertices at even depth is $V$. From the definition of $U$, it follows that for a vertex $u \in U$, the number of children of $u$ is $d$. For vertices $v \in V$, note that the number of children of $v$ is the number of children of $v$ in $G$ divided by $d$. Thus $T$ is a realization of a Poisson $d$-tree. Hence the bijection is proved.

To prove the last part of the lemma, consider a Borel-measurable set $B \subseteq \mathcal{G^*}$. Note that
\begin{align*}
	\nu_{d,\lambda}^{\hat{\varphi}}(B)&= \nu_{d,\lambda}\{T:\hat{\varphi}(T) \in B\}\\
	&=\mathbb{P}\{T:T \text{ is a realization of Poisson } d\text{-tree such that }[\varphi(T)] \in B\},
\end{align*}
where $[\hat{\varphi}(T)]$ is the equivalence class of $\hat{\varphi}(T) $ in $\mathcal{G^*}$. Furthermore, note that $T$ is a Poisson $d$-tree if and only if  $\hat{\varphi}({T})$ is a $d$-block Galton-Watson tree. Thus it follows that
\begin{align*}
	\nu_{d,\lambda}^{\hat{\varphi}}(B)&=\mathbb{P}\{G: G \text{ is a realization of } d\text{-block Galton-Watson graph such that }[T] \in B\}\\
	&=dGW(d\operatorname{Poi} 
	(\lambda))(B).
\end{align*}
This proves that $\nu_{d,\lambda}^{\hat{\varphi}}=dGW(d\operatorname{Poi}(\lambda))$ for each $\lambda > 0$.
\end{proof}
Next, we introduce two lemmas required for the proof of Lemma \ref{lemma: U(Gn) convergence}.
\begin{lemma}\label{lemma: U(G_n)=nu_d,n^phi}
For $d \geq 2,$ and $ n \geq d+1,\ \nu_{d,n}^{\hat{\varphi}}= U(G_n)$ where $U(G_n)$ is as defined in (\ref{eqn: U(G)}).
\end{lemma}
\begin{proof}
Recall that $G_n$ and $\widehat{G}_n$ denote the line graph of $Y_d(n,p)$ and the bipartite graph associated with $Y_d(n,p)$, respectively. 
Consider a realization $X$ of the random simplicial complex $Y_d(n,p)$ and let $G_n^{(X)}$ and $\hat{G}_n^{(X)}$ denote respectively, the line graph and the bipartite graphs corresponding to $X$. For $o \in V_n$, let $G_n^{(X)}(o)$ and $\widehat{G}_n^{(X)}(o)$ denote the connected component of $G_n^{(X)}$ and $\widehat{G}_n^{(X)}$ containing $o$, respectively.

For $v_1,v_2 \in V_n$, note that $(v_1,v_2) \in E_n^{(X)}$ if and only if $v_1 \cup v_2$ is a $d-$cell in $X$. Also, $v_1 \cup v_2 \in X^d$ if and only if both $(v_1,v_1 \cup v_2 )$ and $(v_2,v_1 \cup v_2)$ belong to the edge set of $\widehat{G}_n^{(X)}$. Thus $(v_1,v_2) \in E_n^{(X)}$ if and only if $(v_1,v_2) \in E\left(\varphi(\widehat{G}_n^{(X)})\right)$ where $E\left(\varphi(\widehat{G}_n^{(X)})\right)$ denotes the edge set of $\varphi(\widehat{G}_n^{(X)})$ . As a consequence, 
\begin{equation}\label{eqn:G_n(X),G_n hat (X)}
	\left[(G_n^{(X)}(o),o)\right]=\hat{\varphi}\left[(\widehat{G}_n^{(X)}(o),o)\right].
\end{equation} 
Since (\ref{eqn:G_n(X),G_n hat (X)}) holds for each point of the sample space, we get that
\begin{align}\label{eqn: phi hat equivalence}
\left[\left(G_n(o),o\right)\right]=\hat{\varphi}\left(\left[\widehat{G}_n (o),o\right]\right).
\end{align}
Now, consider $\nu_{d,n}^{\hat{\varphi}}$. By (\ref{eqn:pushforward measure defn}) and (\ref{eqn: phi hat equivalence}), we have that for every Borel measurable set $B \subseteq \mathcal{G}^*$,
\begin{align*}
\nu_{d,n}^{\hat{\varphi}}(B)&= \nu_{d,n}\left(\hat{\varphi}^{-1}(B)\right) \quad \\
&=\frac{1}{|V_n|}\#\{o \in V_n : \hat{\varphi}\left(\left[\widehat{G}_n (o),o\right]\right) \in B\}\\
&=\frac{1}{|V_n|}\#\{o \in V_n : \left[\left(G_n(o),o\right)\right] \in B\} \quad \\
&= U(G_n)(B).
\end{align*}
Since, the equality holds for all measurable sets, the two measures are equal.
\end{proof}
Note that $\nu_{d,n}$ and $U(G_n)$ are random measures. From \cite{linial_peled}, we have the following Lemma about the convergence of $\nu_{d,n}$. In \cite{linial_peled}, Linial and Peled have used a marked version of $\widehat{G}_n$. It is easy to see that the lemma also holds for unmarked version. For details see [Section 2.3, \cite{linial_peled}].
\begin{lemma}\label{lemma: weak limit nu d,n}
For $d \geq 2$ and $p(n)= \lambda/n$ for some $\lambda > 0$, $\nu_{d,n}$ converges weakly to $\nu_{d,\lambda}$ almost surely, where $\nu_{d,\lambda}$ is the probability measure on $\mathcal{G}^*$ induced by Poisson d-tree with parameter $\lambda$, defined in Definition \ref{defn: Poisson d-tree}.
\end{lemma}

\begin{proof}[Proof of Lemma \ref{lemma: U(Gn) convergence}]
Consider a bounded continuous function $f$ on $\mathcal{G}^*$. By Lemma \ref{lemma: U(G_n)=nu_d,n^phi} and (\ref{eqn:change_of_var}) we obtain,
\begin{align*}
\int_{\mathcal{G}^*} f \mathrm{d} U(G_n)&=\int_{\mathcal{G}^*} f \mathrm{d}\nu_{d,n}^{\hat{\varphi}}\\
&=\int_{\widehat{\mathcal{G}}^*} f\circ\hat{\varphi}\ \mathrm{d} \nu_{d,n}.
\end{align*}
By Proposition \ref{prop: phi hat cont}, we have $\hat{\varphi}$ is continuous and since $f$ is bounded and continuous, $f \circ \hat{\varphi}$ is a bounded continuous function. Therefore by Lemma \ref{lemma: weak limit nu d,n} and change of variable formula,
\begin{align*}
\lim_{n \rightarrow \infty}\int_{\widehat{\mathcal{G}}^*} f\circ\hat{\varphi}\ \mathrm{d} \nu_{d,n}&= \int_{\widehat{\mathcal{G}}^*} f\circ\hat{\varphi}\ \mathrm{d}\nu_{d,\lambda}\\
&= \int_{\widehat{\mathcal{G}}^*} f \mathrm{d} \nu_{d,\lambda}^{\hat{\varphi}},
\end{align*}
almost surely.
This implies that for all bounded continuous functions $f$ on $\widehat{\mathcal{G}}^*$,
\begin{equation*}
\lim_{n \rightarrow \infty} \int_{\mathcal{G}^*} f \mathrm{d} U(G_n)=\int_{\widehat{\mathcal{G}}^*} f \mathrm{d} \nu_{d,\lambda}^{\hat{\varphi}}, 
\end{equation*}
almost surely. Thus, $U(G_n)$ converges weakly to $\nu_{d,\lambda}^{\hat{\varphi}}$ almost surely.
\end{proof}

\subsection{Proof of Theorem \ref{thm: spectral measure dGW}}

Recall the definition of unsigned adjacency matrix of a simplicial complex in Definition \ref{defn:adjacency matrix}. Here, we introduce a more general version of the adjacency operator of a graph. Recall that the space $\mathcal{G}^*$ introduced in Section \ref{sec: local_weak} is the space of locally finite connected rooted graphs. For a locally finite graph $G=(V,E)$, define the adjacency operator as the operator on $\ell^2(V)$ given by
\begin{equation*}
A \psi(u)=\sum_{v:\{u, v\} \in E} \psi(v),
\end{equation*}
where $\psi \in \ell^2(V)$ has finite support. 

Under the assumption that the degrees of the graph $G$ is bounded, $A$ admits a unique self-adjoint extension to $\ell^2(V)$.  
For a rooted graph $(G,o)$ with adjacency operator $A$, we define the spectral measure of the root as the unique probability measure on $\R$ such that for all integers $k \geq 1$,
\begin{equation*}
	\int x^k d \mu_A^o=\langle e_o, A^k e_o \rangle ,
\end{equation*}
where $e_o \in \ell^2(V) $ is given by $e_o(u)=\delta_{o}(u)$. It follows that
if two rooted graphs $(G_1,o_1)$ and $(G_2,o_2)$ are isomorphic, then the corresponding spectral measures of the roots, $\mu_{A^{{o_1}}_1}$ and $\mu_{A^{o_2}_2}$, are equal. Thus, the probability measure $\mu_A^{o}$ is well-defined for all $(G,o) \in \mathcal{G^*}$. For a probability measure $\rho \in \mathcal{P}(\mathcal{G^*})$, we define the expected spectral measure of the root as
\begin{equation*}
	\mu_\rho:=\mathbb{E}_\rho\mu_A^{e_0}.
\end{equation*}
It is also possible to extend the construction of $\mu_\rho$ for unimodular measures $\rho$ (For more details see \cite{coursSRG}).

For a finite graph $G$ with $|V|=n$, consider the probability measure $U(G)$ defined in (\ref{eqn: U(G)}). It follows that 
$$\mu_{U(G)}= \frac{1}{n} \sum_{k=1}^n \delta_{\lambda_k},$$
where $\lambda_1, \lambda_2, \ldots , \lambda_n$ are the eigenvalues of the adjacency matrix of $G$. 
For a random simplicial complex $Y_d(n,p)$, consider the line graph $G_n$ of $Y_d(n,p)$ defined in Definition \ref{defn: graph Gn}. 
the expected empirical spectral measure of adjacency matrix of $G_n$. Thus it follows that $\mu_{\mathbb{E}U(G_n)}$ is the expected empirical spectral measure of the unsigned adjacency matrix of $Y_d(n,p)$. 

We use the following proposition from \cite{BVS} to prove Theorem \ref{thm: spectral measure dGW}.
\begin{proposition}[Proposition 1.4, \cite{BVS}]\label{Prop mu_rho converg}
Let $\rho  \in \mathcal{P}(\mathcal{G^*})$ and $\rho_n \in \mathcal{P}(\mathcal{G^*})$ be unimodular. Suppose $\rho_n$ converges to $\rho$ weakly, then $\mu_{\rho_n}$ converges weakly to $\mu_\rho$.
\end{proposition} 

\begin{proof}[Proof of Theorem \ref{thm: spectral measure dGW}]
	By Theorem \ref{thm: weak limit}, we have that when $np \rightarrow \lambda$, $U(G_n)$ converges weakly to the probability measure $dGW(\operatorname{Poi}(\lambda))$ almost surely. Thus, from Proposition \ref{Prop mu_rho converg}, it follows that $\mu_{U(G_n)}$ converges weakly to $\mu_{dGW(\operatorname{Poi}(\lambda))}$, the spectral measure of the $d$-block Galton-Watson graph with offspring distribution $d \operatorname{Poi}(\lambda)$. Furthermore, since $\mu_{U(G_n)}$ is the empirical spectral distribution of $G_n$, from Theorem \ref{thm: existence_LSD} we have that $\mu_{U(G_n)}$ converges weakly to $\Gamma_d(\lambda)$ almost surely. Combining these two observations, we get that $\Gamma_d(\lambda)$ is equal to $\mu_{dGW(\operatorname{Poi}(\lambda))}$ in a set of probability one. Now, we use [Lemma 3.1, \cite{coursSRG}], which states that if $\rho$ is supported on finite graphs, then $\mu_\rho$ is purely atomic. Note that for $\lambda \leq 1/d$, the $d$-block Galton-Watson graph with offspring distribution $d\operatorname{Poi}(\lambda)$ is a finite graph with probability one. This proves Theorem \ref{thm: spectral measure dGW}.
\end{proof}
\bibliographystyle{amsplain}
\bibliography{Spectrum_of_RSC_in_Thermo_Regime}

\end{document}

%% file: mypream.tex
    \newtheorem{theorem}{Theorem}
    \newtheorem{lemma}[theorem]{Lemma}
    \newtheorem{proposition}[theorem]{Proposition}

\theoremstyle{definition} 
    \newtheorem{definition}[theorem]{Definition}

    \newtheorem{remark}[theorem]{Remark}
    \newtheorem{example}[theorem]{Example}
    \newtheorem{exercise}[theorem]{Exercise}

    \newtheorem{notation}[theorem]{Notation}

\renewcommand \qedsymbol {$\blacksquare$}

\def\Z{\mathbb{Z}}

\def\N{\mathbb{N}}

\def\R{\mathbb{R}}

\def\Z{\mathbb{Z}}

\def\sgn{{\mb{sgn}}}

\def\<{\langle}
\def\>{\rangle}

\def\mb{\mbox}

\def\bar{\overline}

\newcommand\Tr{{\mbox{Tr}}}

\newcommand\mnote[1]{} 
\newcommand\be{\begin{equation*}}

\newcommand\ee{\end{equation*}}

\newcommand\ben{\begin{equation}}
\newcommand\een{\end{equation}}
\newcommand\bes{\begin{eqnarray*}}
\newcommand\ees{\end{eqnarray*}}

\newcommand\bex{\begin{exercise}}
\newcommand\eex{\end{exercise}}
\newcommand\beg{\begin{example}}
\newcommand\eeg{\end{example}}
\newcommand\benu{\begin{enumerate}}
\newcommand\eenu{\end{enumerate}}
\newcommand\beit{\begin{itemize}}
\newcommand\eeit{\end{itemize}}
\newcommand\berk{\begin{remark}}
\newcommand\eerk{\end{remark}}
\newcommand\bdefn{\begin{defintion}}
\newcommand\edefn{\end{definition}}
\newcommand\bthm{\begin{theorem}}
\newcommand\ethm{\end{theorem}}
\newcommand\bprf{\begin{proof}}
\newcommand\eprf{\end{proof}}
\newcommand\blem{\begin{lemma}}
\newcommand\elem{\end{lemma}}

\newcommand{\supp}{\mbox{\rm supp}}

\newcommand{\sm}{{\raise0.3ex\hbox{$\scriptstyle \setminus$}}}

\def\mb{\mbox}

\def\CHI{\mathchoice%
{\raise2pt\hbox{$\chi$}}%
{\raise2pt\hbox{$\chi$}}%
{\raise1.3pt\hbox{$\scriptstyle\chi$}}%
{\raise0.8pt\hbox{$\scriptscriptstyle\chi$}}}
\def\smalloplus{\raise1pt\hbox{$\,\scriptstyle \oplus\;$}}